\documentclass[a4paper,10pt]{article} 

\usepackage[hmargin=35mm,vmargin=40mm]{geometry}

\usepackage{latexsym,amsmath,amssymb,a4wide,amscd}
\usepackage[latin1]{inputenc}
\usepackage[T1]{fontenc}

\usepackage{array}

\usepackage{enumerate}

\usepackage{amsthm}

\usepackage{lscape}

\usepackage{url}

\usepackage{mathabx}

\usepackage{longtable}
\usepackage{multirow}

\usepackage{MnSymbol}

\usepackage{mathrsfs}

\usepackage{amsthm}
\numberwithin{equation}{section}
\numberwithin{figure}{section}
\theoremstyle{definition}
\newtheorem {definition}{Definition}
\theoremstyle{plain}
\newtheorem {theorem}{Theorem}
\numberwithin{theorem}{section}
\newtheorem {lemma}[theorem]{Lemma}

\theoremstyle{remark}
\newtheorem {remark}{Remark}

\usepackage[pdftex]{graphicx} 
\usepackage{float}

\usepackage{pdfpages}

\usepackage[nooneline]{caption}

\begin{document}

\date{}

\title{\Large {\bf Combinatorial Seifert fibred spaces with transitive cyclic automorphism group}}

\author{Benjamin Burton and Jonathan Spreer}

\maketitle

\subsection*{\centering Abstract}

{\em
	In combinatorial topology we aim to triangulate manifolds such that their topological 
	properties are reflected in the combinatorial structure of their description.
	Here, we give a combinatorial criterion on when exactly 
	triangulations of $3$-manifolds with transitive cyclic symmetry 
	can be generalised to an infinite family
	of such triangulations with similarly strong combinatorial properties.

	In particular, we construct triangulations
	of Seifert fibred spaces with transitive cyclic symmetry where the symmetry
	preserves the fibres and acts non-trivially on the homology of the spaces.
	The triangulations include the
	Brieskorn homology spheres $\Sigma (p,q,r)$, 
	the lens spaces $\operatorname{L} (q,1)$ 
	and, as a limit case, $(\mathbf{S}^2 \times \mathbf{S}^1)^{\# (p-1)(q-1)}$.	
}\\
\\
\textbf{MSC 2010: } 
{\bf 57Q15};  
57N10; 
05B10; 
20B25; 

\medskip
\textbf{Keywords: } combinatorial topology, (transitive) combinatorial 3-manifold, Seifert fibred space, 
Brieskorn homology sphere, (transitive) permutation group, difference cycle, cyclic 4-polytope, cyclic group

\section{Introduction}
\label{sec:def}

It is the defining goal of combinatorial topology to establish links between the 
combinatorial structure of an object and its topology. Of course, this is not
possible in general since each individual topological object can usually be described
by a large and diverse class of different combinatorial objects, typically with very distinct properties.
Hence the question of how to choose a combinatorial structure which 
describes a topological object ``best'' is of critical importance.

If the right constraints are imposed on the combinatorial structure of an object,
topological properties become transparent which otherwise are hard to obtain. 
For instance a simplicial complex where every triple of vertices spans a triangle 
has to be simply connected \cite{Kuehnel95TightPolySubm}.

In other words, the right choice of combinatorial object makes the topology of 
a manifold combinatorially accessible.

\medskip
In ``non-combinatorial'' (conventional) $3$-manifold topology there are well established methods for describing
manifolds in ways that make their topological structure easily understandable.
One of these methods makes use of the fact that any closed oriented $3$-manifold 
can be obtained from the $3$-sphere by repeatedly applying {\it Dehn surgery}. 
Moreover, there is the standard 
JSJ decomposition \cite{Jaco78JSJDec,Johannson79JSJDec}
for prime $3$-manifolds
where {\em Seifert fibred spaces} come naturally out of the
construction. Seifert fibred spaces
are $3$-manifolds which are obtained by starting with a very restricted and 
well understood class of fibrations of the circle over a surface,
followed by performing surgery parallel to the fibres.

In the combinatorial setting we work with {\em combinatorial manifolds}
which are simplicial complexes with some additional properties.
As a result even the basic form of Dehn surgery needed to construct
Seifert fibred spaces introduces unwanted complexity
because gluing simplicial complexes together 
can require significant and sometimes unwieldy modifications. 
The \texttt{GAP}-script SEIFERT \cite{Lutz08ManifoldPage} by Lutz and Brehm constructs 
arbitrary combinatorial {\em Seifert fibred spaces}. However, due to the added complexity
of the gluings involved, the output complexes of the script 
are typically difficult to analyse.

In this article we aim to overcome this difficulty by explicitly constructing
combinatorial structures that reflect the topological 
properties of the objects we want to represent. More precisely, since Seifert fibrations
are unions of disjoint circles, we focus on combinatorial $3$-manifolds which,
in a certain sense, are invariant under rotations. In more combinatorial terms, 
we are interested in complexes with {\em transitive cyclic symmetries}; that is,
complexes with automorphism groups acting transitively on its vertices.

In addition to the philosophical compatibility of a rotational symmetry with $\mathbf{S}^1$-fibrations,
combinatorial $3$-manifolds with transitive cyclic symmetry have a number of other appealing
properties. They are globally determined by only a local 
neighbourhood, which means that the amount of data needed to describe them is much smaller
than the complex itself. Furthermore, they are easy to construct due to their transitive symmetry, 
and particularly easy to analyse due to the simplicity of the cyclic group. 
As a consequence, this type of combinatorial manifold has been a canonical choice 
for a good representative of the underlying topological manifold in the work of 
many authors over the past decades 
(for instance, see \cite{Brehm09LatticeTrigE33Torus,Kuehnel85NeighbComb3MfldsDihedralAutGroup,
Lutz03TrigMnfFewVertVertTrans,Lutz09EquivdCovTrigsSurf,Spreer14CyclicCombMflds}).

\medskip
In addition to constructing such combinatorial manifolds we are interested in making these
constructions compatible with Dehn surgery.
Of course, working exclusively with combinatorial $3$-manifolds with vertex transitive 
cyclic automorphism group implies even stronger restrictions to performing Dehn surgery than 
the restrictions already present in the general combinatorial setting. As a consequence, despite 
all research about combinatorial manifolds with transitive symmetry, there are only very 
few examples of combinatorial surgery preserving a given transitive cyclic symmetry. 

\begin{itemize}
	\item There is a $14$-vertex triangulation of the $3$-sphere containing two disjoint solid 
		$7$-vertex tori in form of one {\it difference cycle}, i.e., an orbit of the action of 
		the transitive cyclic automorphism group on the triangulation. This difference cycle can be replaced by 
		another difference cycle with equal boundary yielding a triangulation of 
    $\mathbf{S}^2 \times \mathbf{S}^1$ 
		and, in a slightly different setting, a triangulation of the lens space 
		$\operatorname{L} (3,1)$ \cite[Section 4.5.1]{Spreer10Diss}. 
	\item In \cite{Kuehnel85NeighbComb3MfldsDihedralAutGroup} K\"uhnel and Lassmann construct an 
		infinite family of neighbourly $3$-dimensional combinatorial $n$-vertex Klein bottles, 
		$n \geq 9$, using a special property of the boundary complex of the {\it cyclic 
		$4$-polytope} $C_4(n)$: By {\it Gale's evenness condition}, the boundary of the 
		$4$-dimensional cyclic polytope with $n$ vertices $\partial C_4(n)$, $n \geq 9$, 
		can be decomposed into two $n$-vertex solid tori $A(n)$ and $B(n)$. This yields a 
		{\em handlebody decomposition of genus one} of the combinatorial $3$-sphere $\partial C_4 (n)$ 
		respecting the transitive cyclic symmetry (cf. for example \cite[Section 5B]{Kuehnel95TightPolySubm}) 
		and hence provides an excellent starting point to perform Dehn surgery in a
		combinatorial setting with transitive symmetry.
	\item In \cite{Spreer14CyclicCombMflds} a related technique is used to construct a family of 
		infinitely many distinct lens spaces $L_k$: For every $k \geq 0$, a $14+4k$ vertex base complex is 
		glued to two solid tori, this way realising combinatorial surgery in infinitely many distinct ways.
\end{itemize}

We want to exploit the above constructions, and in particular the decomposition of $\partial C_4(n)$, to build 
Seifert fibred spaces where the combinatorics of the complex reflects the topological
structure of the fibration (i.e., combinatorial Seifert fibred spaces with transitive 
cyclic symmetry in which solid tori such as $A(n)$ and $B(n)$ can be plugged in to build neighbourhoods 
of the exceptional fibres).

For example, in the genus one handlebody decomposition of the boundary complex of the 
cyclic $4$-polytope $\partial C_4 (n) = A(n) \cup B(n)$ we replace $A(n)$ by another 
simplicial complex $\tilde{A}(n)$ with transitive cyclic symmetry and equal boundary. 
This gives rise to a closed complex in which $B(n)$ acts as an embedded 
solid torus where the gluing map depends on the number of vertices $n$ and 
the choice of a particular decomposition $\partial C_4 (n) = A(n) \cup B(n)$ 
(cf.\ parameter $l$ in Equations~(\ref{eq:decCycPolI}) and (\ref{eq:decCycPolII})). 

\medskip
Constructing such complexes is not trivial in general, but strongly depends on one 
of the key properties of combinatorial manifolds with transitive symmetry: these 
complexes are easy to find. In \cite{Spreer14CyclicCombMflds} there is a classification
of combinatorial $3$-manifolds with transitive cyclic symmetry up to $22$ vertices. 
Searching this classification for complexes containing a 
solid torus of type $B(n_0)$ (for a fixed $n_0 \leq 24$) resulted in a large number of 
candidates for families of Seifert fibred spaces (the complete list is available from 
the second author upon request). 

Our first main result Theorem~\ref{thm:main} essentially describes a setting where
a single example of ``combinatorial surgery'' can be {\em expanded} into an infinite
family of such examples. Using Theorem~\ref{thm:main}, the candidates above can then be 
checked for whether or not they allow such an expansion to an infinite family of combinatorial 
$3$-manifolds and hence into a candidate for a {\em family} of Seifert fibred 
spaces as described above.

\begin{theorem}
	\label{thm:main}
	Let $M$ be an $n$-vertex combinatorial $3$-manifold, $n$ even, given by $m$ difference 
	cycles $d_1$, $1\leq i \leq m$ and $(1 : n/2-1 : 1 n/2-1)$. Then for all $k \geq 0$, $M$ admits
	an expanded version $M_k$ with $n+k$ vertices if and only if each difference cycle $d_i$ 
	contains an entry greater or equal to $n/2$. If $M$ is neighbourly $M_k$ is neighbourly and vice versa for all $k \geq 0$.
\end{theorem}

The above construction is made more precise and explained in detail in Section~\ref{sec:proof}. 

\medskip
Theorem~\ref{thm:main} describes families of combinatorial $3$-manifolds with transitive 
cyclic symmetry. In the course of this article we show that this construction
is suitable to find expansions of triangulated Seifert fibred spaces with multiple 
exceptional fibres where different levels of expansion, i.e., different values of $k$
in the above description, determine different types of exceptional fibres. This 
provides a more systematic approach for describing combinatorial surgeries like the ones
mentioned above (cf. \cite{Kuehnel85NeighbComb3MfldsDihedralAutGroup,Spreer10Diss,Spreer14CyclicCombMflds})
and allow more complex constructions. In particular, we present the following 
$3$-parameter family of triangulations of Seifert fibred spaces with an
unbounded number of exceptional fibres.

\begin{theorem}
	\label{thm:brieskornSpheres}
	There is a $3$-parameter family $\operatorname{M}(p,q,r)$, $2 \leq p < q$ co-prime, 
	$r > 0$, of combinatorial Seifert fibred spaces with $2pq+r$ vertices and
	transitive cyclic automorphism group of topological type 
	$$ \operatorname{SFS} [ (\mathbb{T}^2 )^{\# (a-1)(b-1)/2} : (-p/a,b_1)^b (q/b,b_2)^a (-r/ab ,b_3) ] $$
	where $(\mathbb{T}^2 )^{\# g}$ is the orientable surface of genus $g$,  
	$(x,y)^{\nu}$ denotes a set of $\nu$ exceptional fibres of type $(x,y)$,
	$a := \operatorname{gcd} (p,r)$, $b := \operatorname{gcd} (q,r)$, and
	$$ (\frac{b_1}{p} - \frac{b_2}{q} + \frac{b_3}{r}) \frac{pqr}{ab} = 1.$$
	The isomorphism type of the Seifert fibration is determined by these 
	conditions and, in particular, we have
	\begin{enumerate}[(i)]
		\item $\operatorname{M}(p,q,r)$ is the Brieskorn homology sphere $\Sigma (p,q,r)$ whenever $p$, $q$ and $r$ are co-prime,
		\item $\operatorname{M}(2,q,2)$ is the lens space $\operatorname{L}(q,1)$ and
	\end{enumerate}
	In the case $r=0$ we do not obtain Seifert fibred spaces but the manifolds $(\mathbf{S}^2 \times \mathbf{S}^1)^{\# (p-1)(q-1)}$.
\end{theorem}

We will see that the difference cycles of $\operatorname{M}(p,q,r)$ already reveal
where the fibres are running within the combinatorial manifold. 
Moreover, by the transitive cyclic symmetry
the analysis of the complexes can be done by only considering a fraction of the actual
complex and with the help of the tools of design theory.

\bigskip
The nice combinatorial structure of the complexes allow us to deduce 
further topological properties of the Seifert fibred spaces. Namely, we can show the following
two results.

\begin{theorem}
	\label{kor:HeegaardGenus}
	$\operatorname{M}(p,q,r)$ is of Heegaard genus at most $(p-1)(q-1)$.
\end{theorem}

\begin{samepage}
\begin{theorem}
	\label{thm:actionAutGroup}
	The automorphism group
	$$G \quad := \quad  \operatorname{Aut}(\operatorname{M}(2,q,2kq)) \quad \cong \quad \mathbb{Z}_{2q(k+2)} , $$
	$q$ prime, $k \geq 0$, acts on the first homology group
	$H_1 (\operatorname{M}(2,q,2kq),\mathbb{Z}) = \mathbb{Z}^{q-1}$ by
	$$ \rho : \quad G \quad \to \quad \operatorname{SL} (q-1,\mathbb{Z}) ; \quad \quad g \quad \mapsto \quad
		\begin{pmatrix} 
			0 & \cdots  & \cdots & 0 & -1 \\
			1 & 0 & \cdots & 0 & 1 \\
			0 & \ddots & \ddots & \vdots & \vdots \\
			\vdots & \ddots & \ddots & 0 & -1 \\
			0 & \cdots & 0 & 1 & 1 \\
		\end{pmatrix}^{g}$$
	where $|\rho(G)| = 2q$.
\end{theorem}
\end{samepage}

For $p$ and $q$ fixed Theorem~\ref{kor:HeegaardGenus} gives us infinite families of $3$-manifolds of 
bounded Heegaard genus. This is interesting, as bounds for the Heegaard genus of a $3$-manifold are 
usually hard to obtain in a purely combinatorial setting. Moreover, we show that this bound is tight
whenever $r \equiv 0 \mod pq$ and for $(p,q,r)=(2,3,\geq 3)$.

\medskip
Theorem~\ref{thm:actionAutGroup} describes an interplay between the automorphism group
of $\operatorname{M}(2,q,2kq)$ for $q$ prime (a combinatorial object) and its 
first homology group (a topological invariant). Intuitively,
a combinatorial manifold should be presented in a way such that 
any symmetry of the combinatorial structure is meaningful
for the topological object. For example for the $d$-dimensional torus 
$$\mathbb{T}^d = \mathbf{S}^1 \times \ldots \times \mathbf{S}^1 $$
we would like to have a triangulation where each symmetry
of the combinatorial object permutes the $\mathbf{S}^1$-components, for a connected
sum of manifolds
$$ M^{\# k} = M \# \ldots \# M $$
we would like the symmetries to act on the direct summands, and so on.

In more general terms, if for a combinatorial manifold $M$ the first homology group 
$H_1 (M,R)$ is a free $R$-module of rank $k$, we would like to have a non-trivial 
representation of the automorphism group $\operatorname{Aut} (M)$ of the form
\begin{equation}
	\label{eq:rho}
	\rho : \operatorname{Aut}(M) \to \operatorname{SL} (k,R) .
\end{equation}

However, as of today, few examples are known where such a non-trivial 
representation exist. Theorem \ref{thm:actionAutGroup} describes an infinite family of
further examples using the complex $\operatorname{M}(2,q,2kq)$ in the case that $q$ is a prime.

\bigskip
Finally, there are many more interesting families of Seifert fibred spaces and using Theorem~\ref{thm:main} more can be found. However, the question whether or not this construction principle is suitable to obtain a significant proportion of {\em all} Seifert fibred spaces with a similar degree of impact of the topology on the combinatorics is unanswered as of today and subject to work in progress.

\section{Preliminaries}
\label{sec:prelim}

\subsection{$3$-manifolds and Seifert fibred spaces}

By work of Moise \cite{Moise51TrigOf3Mflds} it follows that every topological $3$-manifold admits a unique piecewise linear and smooth structure and hence all $3$-dimensional manifolds can be triangulated. As a corollary, it follows that every $3$-manifold $M$ can be decomposed into two {\em handlebodies}, i.e., thickened graphs, which are joined along their boundary surface in order to give $M$. The genus of the boundary surface is said to be the {\em genus of the handlebody decomposition of $M$} and the minimum genus over all handlebody decompositions of $M$ is called the {\em Heegaard genus} of $M$. A modification of this construction results in the observation that every $3$-manifold $M$ can be constructed from the $3$-sphere, by drilling out solid tori and gluing them back such that the meridian of the old solid torus in $M$ is identified with a torus knot of type $(p,q)$ on the boundary of the new solid torus. Such a drilling operation is called {\em Dehn surgery} of type $(p,q)$ (see \cite[Theorem 12.14]{Lickorish} for more about Dehn surgery).

$3$-manifolds can be uniquely decomposed into a connected sum of so-called {\em prime $3$-manifolds} which cannot be represented as a non-trivial connected sum. One important class of prime $3$-manifolds can be described as a fibration of the circle over a $2$-dimensional base orbifold with a finite number of additional Dehn surgeries performed along thickened fibres (note that a thickened fibre is a solid torus). Such a representation is called a {\em Seifert fibred space} and is determined by the base surface, the type of the fibration and a list of (rational) Dehn surgeries along the fibres each specified by a pair of co-prime integers (see \cite{Orlik72SeifertMflds} for more about Seifert fibrations).

\subsection{Combinatorial manifolds}

We can represent manifolds in a purely combinatorial piecewise linear fashion 
using simplicial complexes. For each vertex $v$ in a simplicial complex $C$ we 
refer to the {\em link} of $v$ as the boundary of its simplicial neighbourhood, 
i.e., in the set of all simplices containing $v$ the set of proper faces not 
containing $v$. A {\em combinatorial $d$-manifold} is a pure and abstract 
$d$-dimensional simplicial complex such that each vertex link is a triangulated 
$(d-1)$-sphere with the standard piecewise linear structure. If, in a 
simplicial complex, the link of a vertex $v$ is {\em not} a triangulated 
$(d-1)$-sphere with the standard piecewise linear structure, $v$ is referred 
to as a {\em singular vertex}. A combinatorial $d$-manifold is said to be 
{\em neighbourly}, if the underlying simplicial complex contains all possible 
${n \choose 2}$ edges where $n$ is the number of vertices. A combinatorial 
$d$-manifold always has a piecewise linear structure induced by the simplicial 
complex. In general, this is not true for simplicial complexes homeomorphic to 
a manifold (so-called {\em triangulations of manifolds}) as illustrated by a 
triangulation of Edward's sphere in dimension $5$ in 
\cite{Bjoerner00SimplMnfBistellarFlips}. Hence using the notion of a 
combinatorial manifold is necessary if we want to work with piecewise linear 
manifolds.

However, in dimension $3$ things are simpler -- any two triangulations of the 
same $3$-manifold are equivalent and induce a unique piecewise linear structure 
by Moise's theorem \cite{Moise51TrigOf3Mflds} (cf.\ above), and every 
triangulated $3$-manifold is automatically a combinatorial $3$-manifold. 

In the following sections, we refer to combinatorial $3$-manifolds which are 
homeomorphic to Seifert fibred spaces as {\em combinatorial Seifert fibred 
spaces}.

\subsection{Automorphism groups and difference cycles}

Any abstract simplicial complex and hence any combinatorial manifold $M$ can 
be seen as a combinatorial structure consisting of tuples of elements of 
$V = \{ 0, 1, \ldots , n-1 \}$ where each element of $V$ appears in at least 
one tuple. The elements of $V$ are referred to as the vertices of the complex. 

The {\em automorphism group} $\operatorname{Aut}(M)$ of $M$ is the group 
of all permutations $\sigma \in S_n$ of the vertices of $M$ which do not change 
the complex $M$ as a whole. If $\operatorname{Aut}(M)$ acts transitively on the 
vertices, i.e., if for any pair of vertices $u$ and $v$ there is an automorphism 
$\sigma \in \operatorname{Aut}(M)$ such that $\sigma \cdot u = v$, $M$ is 
called a {\em combinatorial manifold with transitive automorphism group} or 
just a {\em transitive combinatorial manifold}. If a transitive combinatorial 
manifold is invariant under the cyclic $\mathbb{Z}_n$-action 
$v \mapsto v+1 \mod n$ (i.e., if for a combinatorial manifold $M$, possibly 
after a relabelling of the vertices, 
$\mathbb{Z}_n = \langle ( 0, 1, \ldots , n-1 ) \rangle$ is a subgroup of 
$\operatorname{Aut}(M)$), then $M$ is called a {\em cyclic combinatorial 
manifold} (here $\langle ( a , b, c, \ldots ) \rangle$ denotes the permutation 
group generated by the permutation $( a,b,c, \ldots)$ given in cycle notation).

For cyclic combinatorial manifolds we have the following special situation: 
Since the entire complex does not change under a vertex-shift of type 
$v \mapsto v+1 \mod n$, two tuples are in one orbit of the cyclic group 
action if and only if the differences modulo $n$ of its vertices are equal. 
Hence we can compute a system of orbit representatives by just looking 
at the differences modulo $n$ of the vertices of all tuples of the 
combinatorial manifold (cf.\ \cite{Kuehnel96PermDiffCyc}). This motivates the 
following definition.

\begin{definition}[Difference cycle]
	\label{def:diffCycle}
	Let $a_i$, $0 \leq i \leq d$, be positive integers, 
  $ n := \sum_{i=0}^{d} a_i$ and $\mathbb{Z}_n = \langle (0,1, \ldots , n-1) 
  \rangle$. The simplicial complex
	\begin{equation*}
		( a_0 : \ldots : a_{d} ) := \mathbb{Z}_n \cdot \langle 0 , a_0 , \ldots , 
  \Sigma_{i=0}^{d-1} a_i \rangle
	\end{equation*}
	is called a {\em difference cycle of dimension $d$ on $n$ vertices} where 
  $\mathbb{Z}_n \cdot \langle \cdot \rangle$ denotes the $\mathbb{Z}_n$-orbit 
  of $\langle \cdot \rangle$. The number of elements of $(a_0 : \ldots : a_d)$ 
  is referred to as the {\em length} of the difference cycle.
	If a simplicial complex $C$ is a union of difference cycles of dimension 
  $d$ on $n$ vertices and $\lambda$ is a unit of $\mathbb{Z}_n$ such that the 
  complex $\lambda C$ (obtained by multiplying all vertex labels by $\lambda$ 
  modulo $n$) equals $C$, then $\lambda$ is called a {\em multiplier} of $C$.
\end{definition}

Note that for any unit $\lambda \in \mathbb{Z}_n^{\times}$, the complex 
$\lambda C$ is combinatorially isomorphic to $C$. In particular, all 
$\lambda \in \mathbb{Z}_n^{\times}$ are multipliers of the complex 
$\bigcup_{\lambda \in \mathbb{Z}_n^{\times}} \lambda C$ by construction. 
The definition of a difference cycle above is equivalent to the one given in 
\cite{Kuehnel96PermDiffCyc}. 

Throughout this article, we describe {\em cyclic combinatorial manifolds} as a 
set of difference cycles with the implication that we take the union of the 
difference cycles to describe the simplicial complex. In this way, many 
problems dealing with cyclic combinatorial manifolds can be solved in an 
elegant way.

\subsection{Cyclic polytopes and combinatorial exceptional fibres}

The family of {\em cyclic polytopes} is a two parameter family $C_d(n)$ of convex simplicial $d$-polytopes given by the convex hull of $n$ distinct points on the momentum curve 
	$$ \mathbf{x} : \mathbb{R} \to \mathbb{R}^d ; \quad t \mapsto (t,t^2, \ldots , t^d)^T.$$
Cyclic polytopes were first described by Carath\'eodory at the beginning of the $20$th century \cite{Caratheodory07CyclicPolytopes} and have played an important role in polytope theory and combinatorics ever since.

A remarkable property of cyclic polytopes is that their combinatorial structure can be described by {\em Gale's evenness condition} \cite{Gale63EvenessCondition}. Labelling the vertices of the polytope $C_d(n)$ by the integers $V(C_d(n))~=~\{ 0, 1, \ldots , n-1 \}$ for increasing $t$, a $d$-tuple $\Delta \subset V(C_d(n))$ is a facet of $C_d(n)$ if and only if all pairs of vertices in the complement $V(C_d(n)) \setminus \Delta$ are separated by an even number of vertices in $\Delta$.

This has the following consequence in even dimensions $2m$. A $2m$-tuple $\Delta := \langle a_0, \ldots , a_{2m-1} \rangle$ is a facet of $C_{2m}(n)$ if and only if $\Delta + l~:=~\langle a_0 + l \!\!\!\mod n, \ldots , a_{2m-1}  + l \!\!\!\mod n \rangle$ is a facet of $C_{2m}(n)$ for all $l \geq 0$. Hence $C_{2m}(n)$ has an automorphism group $\operatorname{Aut}(C_{2m}(n))$ containing $\mathbb{Z}_n~=~\langle ( 0, 1, \ldots , n-1 ) \rangle$ as a subgroup acting transitively on the vertices. To see this, shift the labels of $\Delta$ and of an arbitrary pair of vertices $\{ v_1, v_2 \} \subset V(C_{2m}(n)) \setminus \Delta$, $v_1 < v_2$, by $n-v_2$. Since $\Delta$ contains an even number of vertices and $\{ v_1, v_2 \}$ is arbitrary, $\Delta + (n-v_2)$ satisfies Gale's evenness condition if and only if $\Delta$ satisfied Gale's evenness condition.

\medskip
By Gale's evenness condition, the vertex labels of a facet of $C_{2m}(n)$ split into sequences
	$$l , (l+1) \!\!\!\mod n, (l+2) \!\!\!\mod n, (l+3) \!\!\!\mod n, \ldots$$
of even length. Consequently, a difference cycle $D$ is contained in $C_{2m}(n)$ if and only if $D$ can be written as a concatenation of sequences of consecutive $1$-entries of odd length followed by a single difference greater than $1$. In the case $2m = 4$, the observations above give rise to the following way to describe $\partial C_4(n)$.
	\begin{equation}
		\label{eq:partialc4n}
		\partial C_4(n) := \bigcup \limits_{i = 1}^{\lfloor \frac{n}{2} \rfloor} \{ (1:i:1:n-2-i) \}.
	\end{equation}
Note that in Equation~(\ref{eq:partialc4n}) all $3$-dimensional difference cycles consisting of sequences of $1$-entries of odd length followed by single entries greater than $1$ are listed. From the viewpoint of $3$-manifold theory this description reveals another interesting property. By a simple collapsing argument we can see that
\begin{equation}
	\label{eq:decCycPolI}
	A(l,n) := \bigcup \limits_{i = 1}^{l} \{ (1:i:1:n-2-i) \}
\end{equation}
as well as
\begin{equation}
	\label{eq:decCycPolII}
	B(l,n) := \bigcup \limits_{i = l+1}^{\lfloor \frac{n}{2} \rfloor} \{ (1:i:1:n-2-i) \}
\end{equation}
are triangulated solid tori for all 
$1 \leq l \leq \lfloor \frac{n-1}{2} \rfloor -1$, thus yielding a 
{\em handlebody decomposition of genus one} of the combinatorial $3$-sphere 
$\partial C_4 (n)$ respecting its transitive cyclic symmetry (cf. for example 
\cite[Section 5B]{Kuehnel95TightPolySubm}). Solid tori like $A(l,n)$, $B(l,n)$ 
and related constructions provide families of distinct pairs of solid tori 
with equal boundary and thus provide an excellent set of starting points to 
perform Dehn surgery in a combinatorial setting. For this reason we refer to 
them as {\em combinatorial exceptional fibres}.

\subsection{rsl-functions}
\label{sec:rslfunc}

One of the principal tools to analyse combinatorial manifolds is the use of a discrete Morse type theory following Kuiper, Banchoff and 
K\"uhnel \cite{Banchoff67CritPntCurvEmbPoly,Banchoff83CritPtsCurvEmbPolyhII,Kuehnel95TightPolySubm,Kuiper71MorseRelations}. In this theory, the discrete analogue of a Morse function is given by a mapping from the set of vertices $V$ of a combinatorial manifold $M$ to the real numbers $\mathbb{R}$ such that no two vertices have the same image, in this way inducing a total ordering on $V$. This mapping can then be extended to a function $ f: M \to \mathbb{R} $  by linearly interpolating the values of the vertices of a face of $M$ for all points inside that face. $f$ is called a {\em regular simplexwise linear function} or {\em rsl-function} on $M$.

A point $x \in M$ is said to be {\em critical} for an rsl-function $f:M \to \mathbb{R}$ if \[H_{\star} (M_x , M_x \backslash \{ x \} , \mathbb{F}) \neq 0 \] where $M_x := \{ y \in M \, | \, f(y) \leq f(x) \}$ and $\mathbb{F}$ is a field. Here, $H_{\star}$ denotes simplicial homology. It follows that no point of $M$ can be critical except possibly the vertices, also, in contrast to classical Morse theory, a point can be critical of multiple indices and with higher multiplicity. More precisely we call a vertex $v$ {\em critical of index $i$ and multiplicity $m$} if $\beta_i (M_v , M_v \backslash \{ v \} , \mathbb{F}) = m$. 

A result of Kuiper \cite{Kuiper71MorseRelations} states that the number of critical points of an rsl-function of $M$ counted by multiplicity is an upper bound for the sum of the Betti numbers of $M$, hence extending the famous Morse relations from the smooth theory to the discrete case. In addition, like in the smooth case the alternating sum over the critical points of index $i$ of {\em any} rsl-function of $M$ counted by multiplicity equals the alternating sum over the Betti numbers of $M$ and thus the Euler characteristic of $M$.

\section{Proof of Theorem \ref{thm:main}}
\label{sec:proof}

Theorem \ref{thm:main} gives a purely combinatorial criterion for when a given 
cyclic combinatorial $3$-manifold can be expanded to an infinite family of 
combinatorial $3$-manifolds and hence to a candidate for a family of 
combinatorial Seifert fibred spaces (of distinct topological types). For 
similar (but different) results about cyclic combinatorial manifolds see 
Theorem 3.1 and Theorem 3.7 in \cite{Spreer14CyclicCombMflds}.

\medskip
Before we proof Theorem~\ref{thm:main} we first introduce some notation to 
make the statement of the theorem more precise:
Let $d_i = (d_i^0 : d_i^1 : d_i^2 : d_i^3)$, $1 \leq i \leq m$, be difference cycles with $n$ 
vertices, $n$ even, where w.l.o.g. $d_i^3 \geq d_i^j$ for all $0 \leq j \leq 2$, $1 \leq i \leq m$, 
and let $d_{i,k}$, $1 \leq i \leq m$, $k \geq 0$, be difference cycles with $n+k$ vertices given by 
$d_{i,k} = (d_i^0 : d_i^1 : d_i^2 : d_i^3+k)$.

\medskip 
Then the $n$-vertex combinatorial $3$-manifold $M$ is given by
	$$M = \left \{ d_1 , \ldots , d_m, \left (1:\frac{n}{2}-1:1:\frac{n}{2}-1 \right ) \right \}.$$ 
Now Theorem~\ref{thm:main} states that for all $k \geq 0$ the combinatorial manifold $M$ has an $n+k$-vertex expansion, noted as 
	$$ M_k = \left \{ d_{1,k} , \ldots , d_{m,k} \right \} \bigcup \limits_{\ell = \frac{n}{2}}^{\lfloor \frac{n+k}{2} \rfloor} \left \{ (1:\ell-1:1:n-\ell-1) \right \},$$
if and only if $d_i^0 + d_i^1 + d_i^2 \leq \frac{n}{2}$ for all $1 \leq i \leq m$.

In addition, given this notation, any combinatorial $3$-manifold of the form $M_k$, that
is, $d_i^0 + d_i^1 + d_i^2 + k \leq \frac{n}{2}$ for all $1 \leq i \leq m$,
is the $k-th$ member of such an expansion series. 

\begin{proof}[Proof of Theorem~\ref{thm:main}]
Let $M$ be a combinatorial $3$-manifold with $n$ vertices given by
	$$M = \left \{ d_1 , \ldots , d_m, \left (1:\frac{n}{2}-1:1:\frac{n}{2}-1 \right ) \right \} ,$$
$d= (d_i^0 : d_i^1 : d_i^2 : d_i^3)$ such that $d_i^0 + d_i^1 + d_i^2 \leq \frac{n}{2}$ for all $1 \leq i \leq m$. 

Throughout the proof we use the following naming convention. 
Instead of identifying the $n$ vertices of $M$
with the elements of $\mathbb{Z}_n$ we use the integers 
$ - \frac{n}{2} + 1, - \frac{n}{2} + 2 , \ldots , \frac{n}{2} - 1$
and $ \pm \frac{n}{2}$ (note that $n$ is even) where the labels coincide with the elements
of $\mathbb{Z}_n$ when taken modulo $n$. The tetrahedra containing vertex $0$ in $M$ are then given by
	\small
	$$ \begin{array}{rl} \bigcup \limits_{i = 1}^{m} & \left \{ \langle 0, d_i^0, d_i^0 + d_i^1 , d_i^0 + d_i^1 + d_i^2 \rangle,
						 \langle - d_i^0, 0, d_i^1, d_i^1 + d_i^2 \rangle, \right .\\
						& \left . \langle - d_i^0 - d_i^1 , - d_i^1, 0, d_i^2 \rangle, 
						 \langle - d_i^0 - d_i^1 - d_i^2, - d_i^1 - d_i^2 , - d_i^2, 0 \rangle \right \} \\
		\cup & \left \{ \langle - \frac{n}{2} + 1, 0, 1, \pm \frac{n}{2} \rangle, \langle -1, 0, \frac{n}{2}-1, \pm \frac{n}{2} \rangle \right \}
	\end{array}  $$
\normalsize
In a similar fashion we name the $n+k$ vertices of $M_k$ by 
$ - \frac{n}{2} + 1, - \frac{n}{2} + 2 , \ldots , \frac{n}{2}, \frac{n}{2} +1 , 
\ldots , \frac{n}{2} + k - 1$ and we identify $-\frac{n}{2} = \frac{n}{2} +k$. 
Then we have for the tetrahedra containing $0$ in $M_k$
	\small
	$$ \begin{array}{rl} \bigcup \limits_{i = 1}^{m} & \left \{ \langle 0, d_i^0, d_i^0 + d_i^1 , d_i^0 + d_i^1 + d_i^2 \rangle,
						 \langle - d_i^0, 0, d_i^1, d_i^1 + d_i^2 \rangle, \right .\\
						& \left . \langle - d_i^0 + d_i^1 , - d_i^1, 0, d_i^2 \rangle, 
						 \langle - d_i^0 - d_i^1 - d_i^2, - d_i^1 - d_i^2 , - d_i^2, 0 \rangle \right \} \\
			\bigcup \limits_{\ell = \frac{n}{2}}^{\lfloor \frac{n+k}{2} \rfloor}
						& \left \{ \langle 0 , 1, \ell, \ell+1 \rangle, 
						\langle -1, 0, \ell -1, \ell \rangle ,
						\langle n-\ell, n-\ell +1 , 0, 1 \rangle, 
						\langle n-\ell -1 , n - \ell -1, 0 \rangle \right \} \\
 			\cup & \left \{ \langle - \frac{n}{2} + 1, 0, 1, \pm \frac{n}{2} \rangle, 
				\langle -1, 0, \frac{n}{2}-1, \pm \frac{n}{2} \rangle \right \}
	\end{array}  $$

\normalsize
In particular note that for the first $m$ difference cycles there is no difference between the tetrahedra containing
$0$ in $M$ and the ones in $M_k$ respectively.

\medskip
Since $M$ and $M_k$ all have a transitive automorphism group, all vertex links within each individual complex are isomorphic and hence it suffices to look at the link of vertex $0$ in order to verify that $M$ or $M_k$ is a combinatorial manifold. Since $(~1~:~\frac{n}{2}-1~:~1~:~\frac{n}{2}~-~1~)$ is part of $M$, we know that the link of vertex $0$ appears as indicated in Figure~\ref{fig:linkM} on the top left hand side, where the rest of the link fills the grey area, and all vertices $v$ in the interior of the grey area are labelled by $v - n$ whenever $v > \frac{n}{2}$ (note that $(~1~:~\frac{n}{2}-1~:~1~:~\frac{n}{2}~-~1~)$ is a short orbit of length $\frac{n}{2}$). Now, if we look at the vertex link of $M_k$, $k > 0$, the fact that $d_i^0 + d_i^1 + d_i^2 \leq \frac{n}{2}$ for all difference cycles $d_i$ together with the labelling convention assures that all vertex labels in the interior of the square surrounding the grey area remain unchanged. Outside the grey area the link grows by $2k$ triangles. By considering that the number of vertices of $M_k$ is $n+k$ it is easy to verify by looking at Figure~\ref{fig:linkM} on the top right (the vertex link of $M_1$) and on the bottom (the vertex link of $M_k$) that the vertex link of $M_k$ is again a sphere for all $k>0$.
	\begin{figure}[h!]
		\begin{center}
			\includegraphics[width=\textwidth]{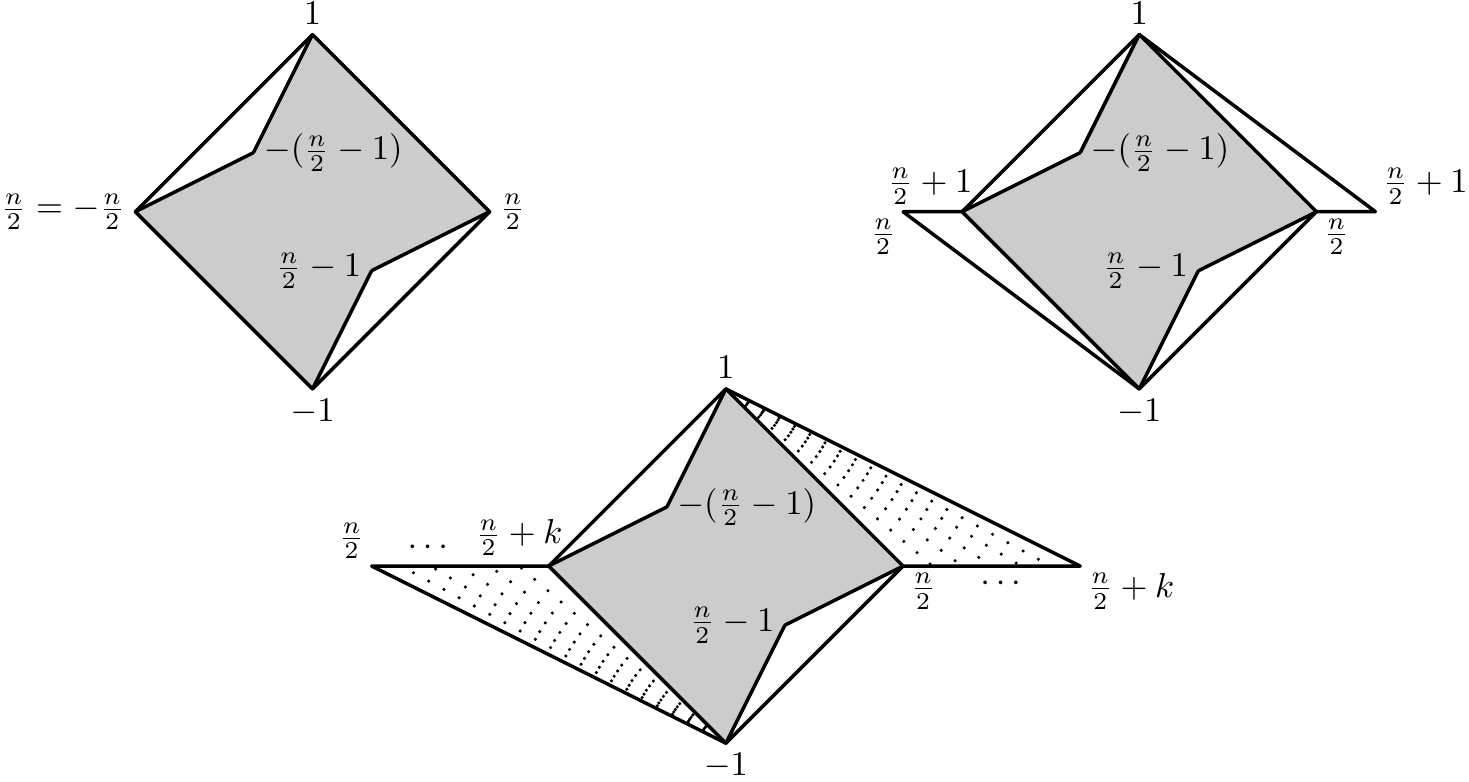} 
			\caption{Link of vertex $0$ of $M$ (top left), $M_1$ (top right) and $M_k$ (bottom). \label{fig:linkM}}
		\end{center}
	\end{figure}

\bigskip
Now assume that for at least one of the difference cycles $d_i$ of $M$ we have $d_i^0 + d_i^1 + d_i^2 > \frac{n}{2}$. If $(1:\frac{n}{2}-1:1:\frac{n}{2}-1)$ is part of $M$ we can write the link of vertex $0$ of $M$ as before (see Figure~\ref{fig:linkM} top left). Now look at the triangle $\langle d_i^0 , d_i^0 + d_i^1 , d_i^0 + d_i^1 + d_i^2 \rangle$. By construction (cf. the first part of the proof), the vertex $d_i^0 + d_i^1 + d_i^2$ is written as $- n + d_i^0 + d_i^1 + d_i^2$ and lies in the interior of the grey area. On the other hand we have $d_i^0 + d_i^1 + d_i^2 = \frac{n}{2} + k_0$ for some integer $k_0 \geq 1$ which lies on the boundary ($k_0 = 1$, see Figure~\ref{fig:linkM} top right) or on the outside ($k_0 > 1$, see Figure~\ref{fig:linkM} on the bottom) of the grey area. Hence the vertex $- n + d_i^0 + d_i^1 + d_i^2 = \frac{n}{2} + k_0$ is singular in the vertex link of $0$ in $M_{k_0}$ and $M_{k_0}$ cannot be a combinatorial manifold.

\bigskip
By the same arguments as presented above, the vertex link of a manifold of the form $M_k$ with $n+k$ vertices must look like the vertex link on the bottom of Figure~\ref{fig:linkM} which thus can be reduced to a manifold of the form $M_0$ with $n$ vertices.

Furthermore, the link of vertex $0$ of $M_k$ contains all vertices $\frac{n}{2}, \frac{n}{2} + 1 , \ldots , n + k$. On the other hand, it contains all vertices $- \frac{n}{2}, - \frac{n}{2} + 1 , \ldots , -1, 1, \ldots , \frac{n}{2}$ if and only if $M$ is neighbourly. Hence $M_k$ is neighbourly if and only if $M$ is neighbourly.
\end{proof}

\begin{remark}
	It seems that infinite series of combinatorial $3$-manifolds as 
	described in Theorem \ref{thm:main} usually contain one further 
	combinatorial $3$-manifold with $n-1$ vertices given by
	$$ M_{-1} = \{ d_{1,-1} , \ldots , d_{m,-1} \} $$
	where $d_{i,-1} := (d_i^0 : d_i^1 : d_i^2 : d_i^3 - 1)$.
	In general, these manifolds then no longer share 
	common difference cycles with the cyclic polytopes. However,
	in many cases the manifolds $M_{-1}$ fit into the
	rest of the family in terms of the topological type.

	The question of whether or not
	such a member $M_{-1}$ always occurs
	or if families can be constructed where $M_{-1}$ is 
	not a combinatorial manifold is interesting but
	has to be left open at this point. 
\end{remark}

\section{A $3$-parameter family of combinatorial $3$-manifolds}
\label{sec:ccSFS}

The aim of this section is to proof Theorems~\ref{kor:HeegaardGenus} and 
\ref{thm:actionAutGroup}, and to prepare the proof of Theorem 
\ref{thm:brieskornSpheres}.

\medskip
Theorem \ref{thm:main} allows us to find large numbers of infinite series of neighbourly combinatorial $3$-manifolds. However, a priori it is not clear which of the families obtained by Theorem \ref{thm:main} actually describe an infinite family of distinct manifolds.
Indeed, existing infinite series of combinatorial $3$-manifolds suggest that most such families consist of infinitely many triangulations of only very few distinct topological $3$-manifolds (cf. \cite[Section 4.5.1]{Spreer10Diss} or \cite{Kuehnel85NeighbComb3MfldsDihedralAutGroup}). Thus to obtain infinite families of interesting $3$-manifolds requires more work. 

The $3$-parameter family of cyclic combinatorial $3$-manifolds given in Theorem~\ref{thm:brieskornSpheres} was constructed by hand, by extending and combining various one-parameter families of interesting combinatorial $3$-manifolds found by applying Theorem \ref{thm:main} and the census of cyclic combinatorial $3$-manifolds from \cite{Spreer14CyclicCombMflds}. The subsequent analysis of the complexes was assisted by computer, using the computational topology software \textsf{simpcomp} \cite{simpcomp,simpcompISSAC,simpcompISSAC11} and the combinatorial recognition routines in \texttt{Regina} \cite{Burton09Regina,Burton12CompTopWRegina}. 

\subsection{Construction of the family}
\label{sec:3paramFam}

In what follows, we construct a $3$-parameter family 
$\operatorname{M}(p,q,r)$ of combinatorial $3$-manifolds with transitive 
cyclic automorphism group, $p$ and $q$ co-prime positive integers, and
$r$ a non-negative integer. $\operatorname{M}(p,q,r)$ 
consists of a base triangulation $\operatorname{B}(p,q,r)$ and, for $r>0$, 
three collections of solid tori $\operatorname{F}_1 (p,q,r)$, 
$\operatorname{F}_2 (p,q,r)$ and $\operatorname{F}_3 (p,q,r)$,
each of which may consist of several solid tori, and each of which has 
compatible boundary with $\operatorname{B}(p,q,r)$. These solid tori are then 
glued to $\operatorname{B}(p,q,r)$ in order to give a closed combinatorial 
$3$-manifold, hence
$$ \operatorname{M}(p,q,r) = \operatorname{B}(p,q,r) \cup 
    \operatorname{F}_1 (p,q,r) \cup 
    \operatorname{F}_2 (p,q,r) \cup 
    \operatorname{F}_3 (p,q,r) . $$
We will see that, for $r>0$, $\operatorname{B}(p,q,r)$ is homeomorphic to a 
bundle over a punctured surface such that the solid tori 
$\operatorname{F}_i (p,q,r)$, $1 \leq i \leq 3$, provide exceptional fibres.

For $r=0$, $\operatorname{F}_3 (p,q,0)$ is not a solid torus but
a collection of $pq$ tetrahedra glued together along common edges.
Nonetheless, $\operatorname{M}(p,q,0)$ is still a combinatorial manifold. 

Recall that we identify the vertices of 
$ \operatorname{M}(p,q,r) $ with the elements of $\mathbb{Z}_{2pq+r}$ and all 
calculations involving the vertex labels are
modulo $2pq+r$. In particular, a vertex denoted by 
$-v$, $pq \leq v \leq 2pq+r$, is interpreted as vertex
$2pq+r-v$ in the naming convention explained in the proof of 
Theorem~\ref{thm:main}.

\medskip
To construct $\operatorname{B}(p,q,r)$, note that $p$ and $q$ are co-prime 
and hence there exist integers $m \in \{ 1, 2, \ldots , q-1 \}$ and 
$k \in \{ 1, 2, \ldots , p-1 \}$ such that $mp - kq  = 1$. The base 
$\operatorname{B}(p,q,r)$ is then given by
%
$$ \operatorname{B}(p,q,r) = \{ (1 \,:\, kq \,:\, (q - m)p \,:\, p q + r),
  (1 \,:\, kq \,:\, p q + r \,:\, (q - m)p),(1  \,:\, p q + r \,:\, kq \,:\, 
  (q - m)p) \}. $$
To construct the first collection of solid tori $\operatorname{F}_1 (p,q,r)$ 
let us assume w.l.o.g. that $(p - k)q > kq$ (if $kq \geq (p-k)q$ the initial 
arguments of the Euclidean algorithm below are interchanged resulting in a 
similar construction).

If the Euclidean algorithm is run with input $kq$ and $(p - k)q$ this yields 
a series of equations
\begin{equation}
    \label{eq:EuclAlgo}
    \begin{array}{lp{0.5cm}lp{0.5cm}l}
	&& a_1 = (p-k)q; && b_1 = kq ;\\
	&& a_2 = a_1 - b_1; && b_2 = b_1 ;\\
	&& \ldots && \ldots \\
	&& \ldots && \ldots \\
	\textrm{if } a_{i} > b_{i}: && a_{i+1} = a_{i} - b_{i}; && b_{i+1} = b_{i}; \\
	\textrm{if } a_{i} < b_{i}: && a_{i+1} = b_{i} - a_{i}; && b_{i+1} = a_{i}; \\
	&& \ldots && \ldots \\
	&& \ldots && \ldots \\
	&& a_{N((p-k)q,kq)} = q ; && b_{N((p-k)q,kq)} = q; \\	
    \end{array}	
\end{equation}
(note that by construction, the greatest common divisor of $kq$ and 
$(p - k)q$ is $q$). Then $\operatorname{F}_1$ is given by
$$ \operatorname{F}_1 (p,q,r) = \{ (b_i \,:\, a_i \,:\, b_i \,:\, 2 p q - 
  2 b_i - a_i + r) \,\,\,|\,\,\,  1 \leq i \leq N((p - k)q,kq) \} .$$
The construction of $\operatorname{F}_2 (p,q,r)$ is analogous. Let w.l.o.g. 
$(q - m)p > mp$. The greatest common divisor of $(q - m)p$ and $mp$ is $p$ 
and if $(c_i,d_i)$, $1 \leq i \leq N((q - m)p,mp)$, is the sequence of integer 
pairs from the Euclidean algorithm as described above then $\operatorname{F}_2$ 
is given by
$$ \operatorname{F}_2 (p,q,r) = \{ (d_i \,:\, c_i \,:\, d_i \,:\, 2 p q - 
  2 d_i - c_i + r) \,\,\,|\,\,\,  1 \leq i \leq N((q - m)p,mp) \} . $$

Finally, the complex $\operatorname{F}_3 (p,q,r)$ is a subset of 
the boundary complex of the cyclic $4$-polytope, namely
$$ \operatorname{F}_3 (p,q,r) = \{ (1 \,:\, p q - 1 + i \,:\, 1 \,:\, 
  p q - 1 + r - i) \,\,\,|\,\,\, 0 \leq i \leq \lfloor r/2 \rfloor + 1 \}, $$
it is a solid torus for $r>0$ and consists of the single short
difference cycle $(1 \,:\, p q - 1 \,:\, 1 \,:\, p q - 1 )$ for $r=0$.

\bigskip
\begin{lemma}
	\label{prop:Mpqr}
	For every pair of co-prime $p$ and $q$, $2 \leq p < q$, and $r \geq 0$,
	the simplicial complex $M(p,q,r)$ is a combinatorial $3$-manifold.
\end{lemma}

\begin{proof}
See Figures~\ref{fig:lk} and \ref{fig:lk1} for drawings of the vertex link 
of vertex $0$ of $\operatorname{M} (p,q,r)$ - a combinatorial $2$-sphere. 
By the transitive symmetry we know that all vertex links are combinatorially 
isomorphic to the link of vertex $0$ and hence all vertex links of 
$\operatorname{M} (p,q,r)$ are homeomorphic to the $2$-sphere.
	\begin{figure}[h!]
		\begin{center}
			\includegraphics[width=\textwidth]{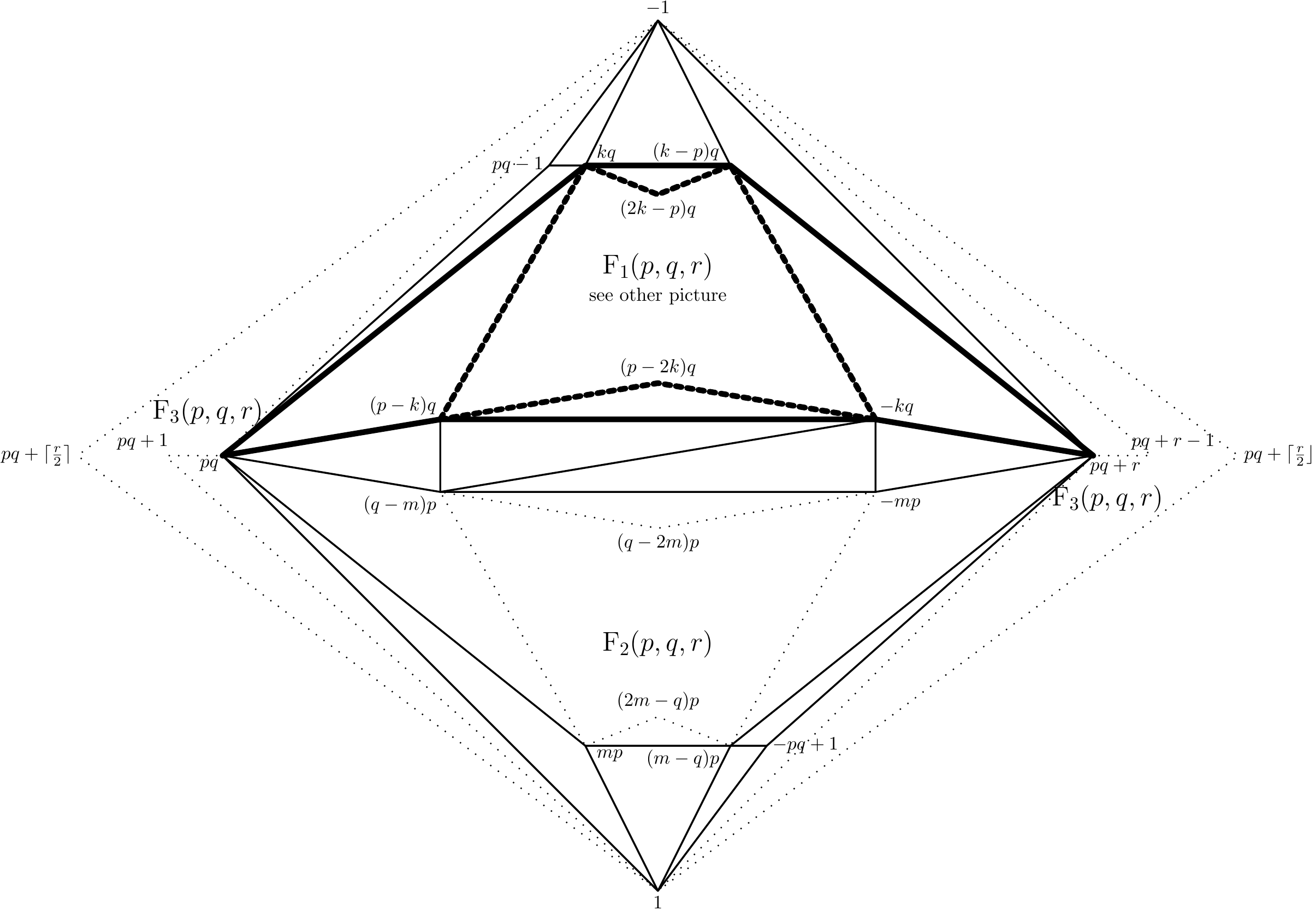} 
			\caption{Vertex link of vertex $0$ of $\operatorname{M} (p,q,r)$. 
      The solid lines represent the triangles belonging to 
      $\operatorname{B} (p,q,r)$, the dotted lines 
			represent the triangles belonging to the first difference cycle of 
			$\operatorname{F}_i (p,q,r)$, $1 \leq i \leq 3$, as indicated.
			See Figure~\ref{fig:lk1} for a more detailed drawing of the region
			$\operatorname{F}_1 (p,q,r)$. \label{fig:lk}}
		\end{center}
	\end{figure}
	\begin{figure}[h!]
		\begin{center}
			\includegraphics[width=.85\textwidth]{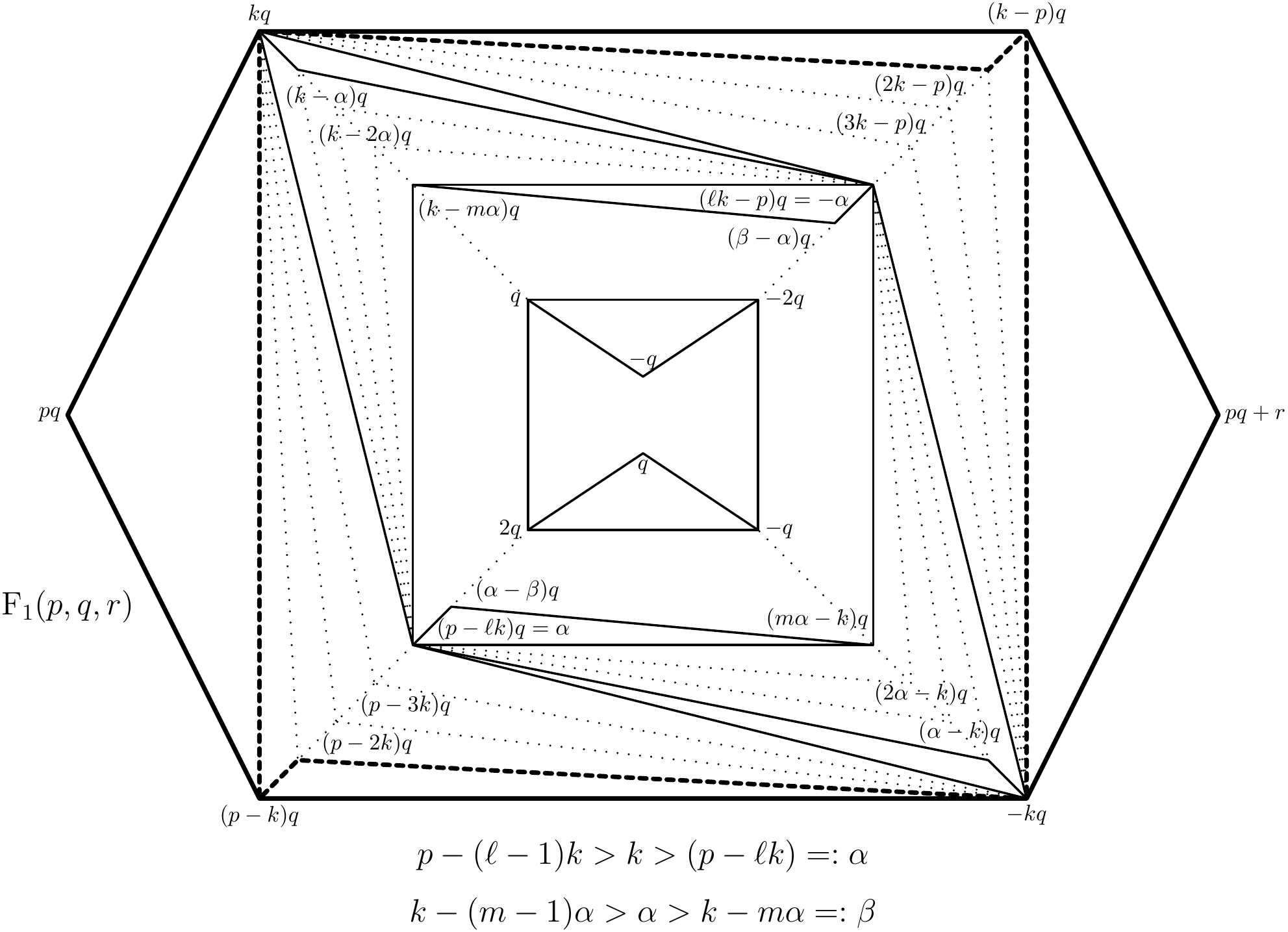} 
			\caption{The disc $\operatorname{F}_1 (p,q,r)$ from Figure~\ref{fig:lk} 
      in more detail.
			\label{fig:lk1}}
		\end{center}
	\end{figure}
\end{proof}

In Section \ref{sec:topTypes} we prove that the combinatorial 
$3$-manifolds $\operatorname{M}(p,q,r)$, $r>0$, are in fact combinatorial 
Seifert fibred spaces with changing topological types and, for $r=0$,
homeomorphic to $(\mathbf{S}^2 \times \mathbf{S}^1)^{\# (p-1)(q-1)}$. 
However, let us first determine 
some other interesting attributes of these combinatorial manifolds.

\subsection{An upper bound for the Heegaard genus of $\operatorname{M}(p,q,r)$}
\label{sec:heegaard}

In this section we determine an upper bound for the Heegaard genus of 
$\operatorname{M}(p,q,r)$ using rsl-functions (cf.\ Section~\ref{sec:rslfunc} 
and \cite{Kuehnel95TightPolySubm}).

\begin{theorem}
	\label{prop:MorseFunc}
	For all $\operatorname{M}(p,q,r)$, $p$ and $q$ co-prime, the rsl-function 
	$$ f : \operatorname{M}(p,q,r) \to [0 , 1] ; \quad v \mapsto \frac{v}{2pq+r-1}$$ 
	has exactly $2(p-1)(q-1)+2$ critical points.
\end{theorem}

In order to prove Theorem~\ref{prop:MorseFunc} we first establish some observations
about critical points of index $1$ of $f$. In doing so we sometimes abuse notation
and refer to a non-critical point as a critical point of index $i$ and multiplicity $0$.
Moreover, the set of faces of a simplicial complex $C$ whose vertices are entirely
contained in a subset $\{ v_1 , \ldots , v_{k} \}$ of the vertices of $C$ are denoted by 
$\operatorname{span}_{\{v_1, \ldots , v_{k} \}} (C)$. Finally, in all of the following calculations
which require the choice of a field we use the field with two elements $\mathbb{F}_2$.

\begin{lemma}
	\label{lem:critIdx1}
	Vertex $v$ of $\operatorname{M}(p,q,r)$, $0 \leq v \leq 2pq + r -1$, is critical of index $1$ and multiplicity
	$$ \tilde{\beta}_0  (\operatorname{span}_{\{-v, -v+1, \ldots , -1 \}} (\operatorname{lk}_{M(p,q,r)} (0))) $$
	with respect to $f$, where $\tilde{\beta}_0 = \beta_0-1$ is the reduced Betti number of index $0$ denoting the number of 
	connected components minus $1$.
\end{lemma}

\begin{proof}
	The multiplicity of a critical point $v$ of index $i$ with respect to an rsl-function $g : \operatorname{M}(p,q,r) \to [0 , 1]$
	is given by the dimension of the $i$-th relative homology $\dim_{\mathbb{F}_2} H_i (M_v,M_v \setminus \{ v \},\mathbb{F}_2)$ where 
	$M_v := \{ x \in \operatorname{M}(p,q,r) \,|\, g(x) \leq g(v) \}$.

	This is equivalent to looking at the $(i-1)$-th reduced Betti number 
	$\tilde{\beta}_{i-1}$ of $\operatorname{span}_{V_v} (\operatorname{lk}_{\operatorname{M}(p,q,r)} (v))$
	where $V_v$ is the subset of vertices $w$ such that $g(w) < g(v)$. 
	For the rsl-function $f (v) = \frac{v}{2pq+r-1}$, $v \in \{ 0 , 1, \ldots , 2pq+r-1 \}$, this means that
	vertex $v$ is critical of index $1$ with multiplicity 
	$\tilde{\beta}_{0} (\operatorname{span}_{\{ 0 , 1, \ldots , v-1 \}} (\operatorname{lk}_{\operatorname{M}(p,q,r)} (v)))$,
	and since $\operatorname{M}(p,q,r)$ has a vertex transitive cyclic automorphism group
	we have $\operatorname{span}_{\{ 0 , 1, \ldots , v-1 \}} (\operatorname{lk}_{M(p,q,r)} (v)) \cong
	\operatorname{span}_{\{ - v, \ldots , -1 \}} (\operatorname{lk}_{M(p,q,r)} (0))$ 
	which proves the result.
\end{proof}

\begin{lemma}
	\label{lem:switchLinks}
	If vertex $-v$ of $\operatorname{M}(p,q,r)$ is not contained in $\operatorname{lk}_{\operatorname{M}(p,q,r)}(0)$, then 
	vertex $v$ is critical of the same index with the same multiplicity as vertex $v-1$ with respect to $f$.
\end{lemma}

\begin{proof}
	If $-v \not \in \operatorname{lk}_{\operatorname{M}(p,q,r)}(0)$, then 
	$$\operatorname{span}_{\{ -v+1, \ldots , -1 \}} (\operatorname{lk}_{\operatorname{M}(p,q,r)} (0)) = \operatorname{span}_{\{ -v, \ldots , -1 \}} (\operatorname{lk}_{\operatorname{M}(p,q,r)} (0))$$
	and hence 
	$$\operatorname{span}_{\{ 1, \ldots , v-2 \}} (\operatorname{lk}_{\operatorname{M}(p,q,r)} (v-1)) = \operatorname{span}_{\{ 1, \ldots , v-1 \}} (\operatorname{lk}_{\operatorname{M}(p,q,r)} (v)) .$$
\end{proof}

\begin{lemma}
	\label{lem:connFibres}
	The complex $\operatorname{span}_{\{ - v, \ldots , -1 \}} 
  (\operatorname{lk}_{\operatorname{F}_i(p,q,r)} (0))$, $1 \leq i \leq 2$, 
  is connected for all integers $-pq\leq-v\leq-1$.
\end{lemma}

\begin{proof}
	We prove Lemma~\ref{lem:connFibres} for $\operatorname{F}_1(p,q,r)$. 
  The proof that $\operatorname{span}_{\{ -v, \ldots , -1 \}} 
  (\operatorname{lk}_{\operatorname{F}_2(p,q,r)} (0))$ is connected for 
  $-pq\leq-v\leq-1$ is completely analogous.

	\medskip
	Recall that
	$$ \operatorname{F}_1 (p,q,r) = \{ (b_i \,:\, a_i \,:\, b_i \,:\, 2 p q - 
    2 b_i - a_i + r) \, |\,  1 \leq i \leq N((p - k)q,kq) \} ,$$
	where the $a_i$ and $b_i$ are given by the Euclidean algorithm.

	Due to the symmetry in the difference cycles of $\operatorname{F}_1$,
	$\operatorname{span}_{\{ - v, \ldots , -1 \}} 
  (\operatorname{lk}_{\operatorname{F}_1(p,q,r)} (0))$ is connected if and 
  only if	$\operatorname{span}_{\{ 1, \ldots , v \}} 
  (\operatorname{lk}_{\operatorname{F}_1(p,q,r)} (0))$ is connected. Hence 
  we focus on the latter and $1 \leq v \leq pq$.
	
	\medskip
	All vertices of 
	$\operatorname{span}_{\{ 1, \ldots , pq \}} 
  (\operatorname{lk}_{\operatorname{F}_1(p,q,r)} (0))$ 
	are of the form $b_i$, $a_i$, $a_i + b_i$ or $a_i + 2b_i$ and the edges are
	of the form $\langle b_i , a_i + b_i \rangle$, $\langle a_i , a_i + b_i \rangle$ 
	or $\langle a_i + b_i , a_i + 2b_i \rangle$ for some $i$, $1 \leq i \leq N((p - k)q,kq)$.
	
	We have $a_i+b_i = \max \{ a_{i-1} , b_{i-1} \}$ (this follows from one step of the Euclidean algorithm given by Equation~(\ref{eq:EuclAlgo})),
	and $a_i+2b_i = \max \{ a_{i-2} , b_{i-2} \}$ which can be seen by considering the following four cases.

	\begin{itemize}
		\item {\bf Case $a_{i-2} - b_{i-2} > b_{i-2}$:} We have $a_i = a_{i-2} - 2 b_{i-2}$ and $b_i = b_{i-2}$
			and the statement follows.
		\item {\bf Case $a_{i-2} > b_{i-2}$ and $a_{i-2} - b_{i-2} < b_{i-2}$:} This results in $a_{i-1} = a_{i-2} - b_{i-2}$ and
			$b_{i-1} = b_{i-2}$ followed by swapping the variables yielding $a_{i} = b_{i-1} - a_{i-1} = 2 b_{i-2} - a_{i-2}$
			and $b_i = a_{i-2} - b_{i-2}$, and hence $a_i + 2 b_i = a_{i-2}$.
		\item {\bf Case $a_{i-2} < b_{i-2}$ and $b_{i-2} - a_{i-2} > a_{i-2}$:} Here we first swap variables, thus, 
			$a_{i-1} = b_{i-2} - a_{i-2}$ and $b_{i-1} = a_{i-2}$ followed by $a_i = b_{i-2}- 2 a_{i-2}$
			and $b_i = a_{i-2}$ and all together $a_i + 2 b_i = b_{i-2}$.
		\item {\bf Case $a_{i-2} < b_{i-2}$ and $b_{i-2} - a_{i-2} < a_{i-2}$:} Now we have to swap variables twice resulting in
			$a_{i-1} = b_{i-2} - a_{i-2}$, $b_{i-1} = a_{i-2}$ and $a_i = 2 a_{i-2} - b_{i-2}$, $b_i = b_{i-2} - a_{i-2}$
			and hence $a_i + 2 b_i = b_{i-2}$.
	\end{itemize}

	Finally for $i \leq 2$ it follows that $a_1+2b_1 = pq$, and $a_2+2b_2 = a_1 + b_1 = \max \{ (p-k)q, kq \} $. As a result we have
	$$\{ a_1, b_1 , a_2, b_2, \ldots , a_{N((p - k)q,kq)}, b_{N((p - k)q,kq)} \} = \{ \max \{a_1,b_1\} , \ldots , \max \{a_{N((p - k)q,kq)}, b_{N((p - k)q,kq)}\} \},$$
	where $\max \{a_i,b_i\} > \max \{a_{i+1},b_{i+1}\} $. 

	It follows that $\operatorname{lk}_{\operatorname{F}_1(p,q,r)} (0)$ contains edges of the form 
	$\langle \max \{ a_{i-1}, b_{i-1} \} , \max \{ a_{i-2}, b_{i-2} \} \rangle$ for all 
	$1 \leq i \leq N((p - k)q,kq)$. Hence there exist a path meeting all vertices of 
	$\operatorname{span}_{\{ 1, \ldots , pq \}} (\operatorname{lk}_{\operatorname{F}_1(p,q,r)} (0))$ 
	in increasing / decreasing order. By symmetry this also holds for vertices $-pq \leq v \leq -1$, and 
	$\operatorname{span}_{\{ -v, \ldots , -1 \}} (\operatorname{lk}_{\operatorname{F}_1(p,q,r)} (0))$ 
	is connected for all $-pq \leq -v \leq -1$.
\end{proof}

\begin{lemma}
	\label{lem:connComplex}
	The complex $\operatorname{span}_{\{ - v, \ldots , -1 \}} (\operatorname{lk}_{M(p,q,r)} (0))$ is connected for all $-2pq -r + 1 \leq -v \leq -pq $.
\end{lemma}

\begin{proof}
	By looking at Figure~\ref{fig:lk} we can see that 
	$\operatorname{span}_{\{ - v, \ldots , -1 \}} (\operatorname{lk}_{\operatorname{M}(p,q,r)} (0))$ 
	is connected for all $-pq-r \leq -v \leq -pq$ and 
	$\operatorname{span}_{\{ - v, \ldots , -1 \}} (\operatorname{lk}_{\operatorname{B}(p,q,r)} (0))$ 
	is connected for all $-2pq-r+1 \leq -v \leq -pq$. Moreover, from the proof of Lemma~\ref{lem:connFibres}
	we can see that both $\operatorname{span}_{\{ -v, \ldots , -1 \}} (\operatorname{lk}_{\operatorname{F}_i(p,q,r)} (0))$,
	$1 \leq i \leq 2$, are connected for $-2pq -r +1 \leq -v \leq -pq$ and attached to
	$\operatorname{span}_{\{ - v, \ldots , -1 \}} (\operatorname{lk}_{\operatorname{B}(p,q,r)} (0))$.
	All together it follows that 
	$\operatorname{span}_{\{ - v, \ldots , -1 \}} (\operatorname{lk}_{\operatorname{M}(p,q,r)} (0))$
	is connected for $-2pq -r +1 \leq -v \leq -pq$.
\end{proof}

\begin{proof}[Proof of Theorem~\ref{prop:MorseFunc}]
	Since $\operatorname{M} (p,q,r)$ contains all edges $\langle v, v+1 \rangle$, $0 \leq v \leq 2pq + r -2$,
	and $\operatorname{M} (p,q,r)$ is a combinatorial manifold, $f$ has exactly one critical vertex of
	index $0$ and exactly one critical vertex of index $3$.

	Now, by Lemma~\ref{lem:critIdx1}, the critical vertices of index $1$ of $f$ 
	and their multiplicities can be determined by
	counting the number of connected components minus $1$ of 
	$$ \operatorname{span}_{\{ - v, \ldots , -1 \}} (\operatorname{lk}_{M(p,q,r)} (0)) $$
	for all $1 \leq v \leq 2pq+r-1$. By Lemma~\ref{lem:switchLinks} and Lemma~\ref{lem:connComplex}, 
	$ \operatorname{span}_{\{ - v, - v+1, \ldots , -1 \}} (\operatorname{lk}_{M(p,q,r)} (0)) $ 
	is connected for $v \geq pq$ and so no vertex $v \geq pq$ can be a critical vertex of $f$ of index $\geq 1$. 

	Furthermore, by Lemma~\ref{lem:switchLinks}
	and Lemma~\ref{lem:connFibres} for each vertex $p \leq v < pq$ the complex 
	$\operatorname{span}_{\{ -v, \ldots , -1 \}} (\operatorname{lk}_{\operatorname{M}(p,q,r)} (0))$ has at most
	$3$ connected components, and hence $v$ is critical of index $1$ with multiplicity at most $2$.

	\medskip
	{\bf Case 1}: Let $(q-m)p > mp$. Recall that $mp = kq +1$ and thus $(p-k)q > kq$. It follows that
	for all vertices $p \leq v \leq q-1$ the complex 
	$ \operatorname{span}_{\{ - v, - v+1, \ldots , -1 \}} (\operatorname{lk}_{M(p,q,r)} (0)) $ 
	has two connected components (cf. Figure~\ref{fig:lk}) and hence these vertices are critical of index $1$ 
	and multiplicity $1$ (cf. Lemma~\ref{lem:switchLinks}), for vertices $q \leq v \leq kq$ we can see that
	$ \operatorname{span}_{\{ - v, - v+1, \ldots , -1 \}} (\operatorname{lk}_{M(p,q,r)} (0)) $
	has three connected components and thus we have critical vertices of index $1$ and multiplicity 
	$2$. For $mp \leq v \leq (q-m)p$ we again have two connected components and thus
	critical vertices of index $1$ and multiplicity $1$ and for all other vertices the complex is connected
	(cf. Lemma~\ref{lem:connComplex}). All together there are $(p-1)(q-1)$ critical points of index $1$.

	\medskip
	{\bf Case 2}: Let $(q-m)p \leq mp$. The same argument as before shows that 
	vertices $p \leq v \leq q-1$ are critical of index $1$
	and multiplicity $1$, vertices $q \leq v \leq (q-m)p$ are critical of index $1$ and multiplicity 
	$2$ and vertices $(p-k)q \leq v \leq kq$ are critical of index $1$ and multiplicity $1$,
	which also results in $(p-1)(q-1)$ critical points of index $1$.

	\medskip
	Now, since the alternating sum over all critical points counted by multiplicity equals the Euler 
	characteristic (which has to be $0$), and $f$ has only one critical point of index $0$ and $3$ each, 
	the number of critical points of index $1$ must equal the number of critical points of index $2$. 
	All together $f$ has $(p-1)(q-1)$ critical points of index
	$1$ and $2$ each and thus $2 (p-1)(q-1) + 2$ critical points in total which proves the result.
\end{proof}

Theorem~\ref{kor:HeegaardGenus} now follows as a simple corollary of
Theorem~\ref{prop:MorseFunc}.

\subsection{The homology groups of $\operatorname{M}(p,q,r)$}
\label{sec:homology}

By the proof of Theorem~\ref{prop:MorseFunc} the rsl-function
$$ f : \operatorname{M}(p,q,r) \to [0 , 1] ; \quad v \mapsto \frac{v}{2pq+r-1}$$
has $(p-1)(q-1)$ critical points of index $1$. Furthermore, Lemma~\ref{lem:connComplex} together
with the transitive cyclic symmetry of $\operatorname{M}(p,q,r)$ tells us that only vertices
$1 < v < pq$ can be critical of index $1$ and again by the transitive cyclic symmetry it follows
that all these critical points of index $1$ have to pair with critical vertices $v \geq pq$.
All together it follows that $B_{-} = \operatorname{span}_{0 ,1 , \ldots , pq-1} (\operatorname{M}(p,q,r))$
and $B_{+} = \operatorname{span}_{pq,pq+1 , \ldots ,2pq+r-1} (\operatorname{M}(p,q,r))$ must be
``handlebodies''\footnote{$B_{-}$ and $B_{+}$ might contain
isolated edges and triangles and are thus only homotopic to a handlebody. However,
there is always a small neighbourhood of $B_{-}$ and $B_{+}$ which is a proper handlebody.}
of genus $(p-1)(q-1)$.

Thus the topological type of $\operatorname{M}(p,q,r)$ is determined by how a set of
$(p-1)(q-1)$ simple closed curves in $B_{-}$ forming a basis of the first homology group
is glued to $B_{+}$.

\medskip
A basis of the first homology group of $B_{-}$ can be found by the observations
made in the previous section (in particular, cf. Lemma~\ref{lem:critIdx1} and the proof of
Theorem~\ref{prop:MorseFunc}): for all $v$, $p \leq v < pq$, we
connect distinct connected components of
$\operatorname{span}_{\{ -v, \ldots , -1 \}} (\operatorname{lk}_{\operatorname{M}(p,q,r)} (0))$
(whenever they exist) by a path in 
$\operatorname{span}_{\{ -v, \ldots , -1 \}} (\operatorname{M}(p,q,r))$.

One possible choice for such a basis of $H_1 (B_{-})$ is
$$ \langle v , v-1, v-2, \ldots , v-p, v \rangle $$
for $p \leq v \leq kq$ and 
$$ \langle v , v-1, v-2, \ldots , v-q, v \rangle $$
for $q \leq v \leq (q-m)p$.

\begin{lemma}
	\label{lem:homology}
	Let $ [c] $ be an element of $H_1 (\operatorname{M}(p,q,r))$. Then
	$$ c \simeq c + pq ,$$
	where for any path $c = \langle v_1, v_2, \ldots , v_r \rangle$ and $x \in \mathbb{Z}_n$, the sum $c + x$ denotes the 
	path obtained from $c$ by adding $x$ mod $n$ component-wise, that is,
	$c + x = \langle (v_1 + x) \mod n, (v_2 + x) \mod n, \ldots , (v_r + x) \mod n \rangle$.
\end{lemma}

\begin{proof}
	We show that $c \simeq c+pq$ for all basis elements in $H_1 (B_{-})$ and hence
	for a generating system of $H_1 (\operatorname{M}(p,q,r))$.

	\medskip
	This is done by using triangles contained in the difference cycles of $\operatorname{F}_1(p,q,r)$ and
	$\operatorname{B}(p,q,r)$ to gradually transform $\langle v, v-1, v-2, \ldots , v-p, v \rangle$ into 
	$\langle pq+v, pq+v-1, pq+v-2, \ldots , pq+v-p, pq+v  \rangle$. 
	The proof for generating elements
	of the form $\langle v, v-1, v-2, \ldots , v-q, v \rangle$ is analogous.

	Let $m = N((q - m)p,mp)$ and $c_i := \max \{ a_{m-i} , b_{m-i} \}$. Then
	\small
	$$ \begin{array}{rcl}
		\langle v, v-1, \ldots , v-p, v \rangle &\simeq& \langle v, pq+v, v-1, pq+v-1, v-2, pq+v-2, \ldots , pq +v-p+1, v-p, v \rangle \\
		&\simeq& \langle v, pq+v, v-1, pq+v-1, v-2, pq+v-2, \ldots , pq+v-p+1, pq+v-p, v-p, v \rangle \\
		&\simeq& \langle v, pq+v, pq+v-1, pq+v-2, \ldots , pq+v-p+1, pq+v-p, v-p, v \rangle \\
		&\simeq& \langle v, pq+v, pq+v-1, pq+v-2, \ldots , pq+v-p, v-c_1, v+c_1, v \rangle \\
		&\simeq& \langle v, pq+v, \ldots , pq+v-p, v-c_1, v+c_1, v+c_2, \ldots , v+c_m, v \rangle \\
		&\simeq& \langle v, pq+v, \ldots , pq+v-p, v-c_1, c_1+v, \ldots , c_m+v, \max \{ (q - m)p,mp \}+v , pq+v, v \rangle \\
		&\simeq& \langle v, pq+v, \ldots , pq+v-p, v-c_1, pq+v-p, pq+v, v \rangle \\
		&\simeq& \langle pq+v, \ldots , pq+v-p, v-c_1, pq+v-p, pq+v \rangle \\
		&\simeq& \langle pq+v, \ldots , pq+v-p, pq+v \rangle . \hfill \qedhere
	\end{array} $$
	\normalsize
\end{proof}

Together with the cyclic symmetry, the above observation allows us to analyse $\operatorname{M}(p,q,r)$
in further detail. In particular, it follows from Lemma~\ref{lem:homology} that
given $p$ and $q$, the homology of $\operatorname{M}(p,q,r)$ only depends on $r\mod pq$.

As a special case, if $r \equiv 0 \mod pq$ we can deduce that $H_1(\operatorname{M}(p,q,r)) = \mathbb{Z}^{(p-1)(q-1)}$,
and if $\operatorname{gcd} (p,r) = \operatorname{gcd} (q,r) = 1$ then all generators of $H_1(B_{-})$ are identified in
$H_1(\operatorname{M}(p,q,r))$ eventually resulting in trivial homology. More generally, if $a = \operatorname{gcd} (p,r)$
and $b = \operatorname{gcd} (q,r)$ then
\begin{equation}
	H_1 (\operatorname{M}(p,q,r)) = \mathbb{Z}^{(a-1)(b-1)} \oplus \mathbb{Z}_{p/a}^{b-1} \oplus \mathbb{Z}_{q/b}^{a-1}
\end{equation}
and since all $\operatorname{M}(p,q,r)$ are orientable we have
$$ H_{\star} (\operatorname{M}(p,q,r)) = (\mathbb{Z},\mathbb{Z}^{(a-1)(b-1)} \oplus \mathbb{Z}_{p/a}^{b-1} \oplus \mathbb{Z}_{q/b}^{a-1},\mathbb{Z}^{(a-1)(b-1)},\mathbb{Z}).$$
We do not prove the claims made above
since they independently follow from the topological types of 
$\operatorname{M}(p,q,r)$
shown in Section~\ref{sec:topTypes}. However, the specific structure of 
$\operatorname{M}(p,q,r)$ given by Theorem~\ref{prop:MorseFunc} and
Lemma~\ref{lem:homology} gives rise to an interesting and rarely observed
connection between the automorphism group of $\operatorname{M}(p,q,r)$ in the
case $p=2$, $q$ prime, $r \equiv 0 \mod pq$, and its first homology group
which is discussed in the following section.

\subsection{Action of the automorphism group on the homology of $\operatorname{M}(2,q,2kq)$}

In this section we present a number of non-trivial group representations of the
cyclic group
$$\operatorname{Aut} (\operatorname{M}(2,q,2kq)) = \langle g \rangle $$
with $g = (0,1, \ldots , 2q(k+2))$, $q$ prime, into the free $\mathbb{Z}$-module
$$ H_1 (\operatorname{M}(2,q,2kq)) = \mathbb{Z}^{q-1} ,$$
$k \geq 0$. In particular, we give a proof of Theorem~\ref{thm:actionAutGroup}.

This is done by applying Lemma~\ref{lem:homology} to a suitable choice 
of a basis of $H_1 (\operatorname{M}(2,q,2kq))$ and following the construction of
finite order integer matrices as described in \cite{Koo03FiniteOrderIntegerMatrices}.

\begin{proof}[Proof of Theorem~\ref{thm:actionAutGroup}]
	Note that by Lemma~\ref{lem:homology} for every cycle $c$ in $\operatorname{M}(2,q,2kq)$
	we have $c \simeq c + 2q$. Hence the size of the image of every action 
	$$ \rho : \operatorname{Aut} (\operatorname{M}(2,q,2kq)) \to \operatorname{SL}(q-1,\mathbb{Z})$$
	divides $2q$ and in particular 
	$$| \rho ( \operatorname{Aut} (\operatorname{M}(2,q,2kq)))| \leq 2q .$$
	In particular, $\rho(g)$ is an integer matrix of order $\leq 2q$.

	Following the observations made in the last section, a basis of 
	$H_1 (\operatorname{M}(2,q,2kq))$ is given by the cycles
	$$ a_{v-1} = \langle v , v-1, v-2, v \rangle $$
	for $2 \leq v \leq q$. Thus by construction we have
	$$ g \cdot a_{i} = a_{i + 1}, $$
	$2 \leq i < q$, where $g$ acts on the cycles of $\operatorname{M}(2,q,2kq)$ by adding $1$ modulo $2q(k+2)$
	to each entry of the cycle.

	Moreover, up to similarity transformations, the only matrix $M \in \operatorname{SL}(q-1,\mathbb{Z})$ of finite 
	order $2q$ is of the form
	$$ M = 
		\begin{pmatrix} 
			0 & \cdots  & \cdots & 0 & -1 \\
			1 & 0 & \cdots & 0 & 1 \\
			0 & \ddots & \ddots & \vdots & \vdots \\
			\vdots & \ddots & \ddots & 0 & -1 \\
			0 & \cdots & 0 & 1 & 1 \\
		\end{pmatrix}$$
	See \cite{Koo03FiniteOrderIntegerMatrices} where $M$ is described in more detail.
	As a side note, in \cite{Brehm09LatticeTrigE33Torus} a similar finite-order integer matrix (of order $q$) occurs in 
	a construction of $d$-dimensional combinatorial tori. 

	Note that the first $q-2$ columns of $M$
	are compatible with the above choice for a basis of $H_1 (\operatorname{M}(2,q,2kq))$
	and, in order to prove Theorem~\ref{thm:actionAutGroup}, it remains to show that
	$$ g \cdot a_{q-1} = a_2^{-1} a_3 a_4^{-1} a_5 \ldots a_{q-2}^{-1} a_{q-1} .$$

	\bigskip
	We have
	$$ \begin{array}{lcl}
		\operatorname{M}(2,q,2kq) &= \,\, \{ & (1 \,:\, q \,:\, q-1 \,:\, 2q(k+1)), \\
		&& (1 \,:\, q \,:\, 2q(k+1) \,:\, q-1), \\
		&& (1 \,:\, 2q(k+1) \,:\, q \,:\, q-1), \\
		&& (1 \,:\, 2q-1 \,:\, q-1 \,:\, 2q(k+1)-1)), \\
		&& (q-1 \,:\, 2 \,:\, q-1 \,:\, 2q(k+1) )), \\
		&& (2 \,:\, q-3 \,:\, 2 \,:\, 2q(k+2) -q-1 )), \\
		&& (2 \,:\, q-5 \,:\, 2 \,:\, 2q(k+2) -q+1 )), \\
		&& (2 \,:\, q-7 \,:\, 2 \,:\, 2q(k+2) -q+3 )), \\
		&& \ldots , \\
		&& (2 \,:\, 2 \,:\, 2 \,:\, 2q(k+2) -6 )) \,\, \} .
	\end{array} $$ 
	In particular, we have the following triangle relations:
	$$ \begin{array}{lcl}
		(1 \,: \, q) & \leftrightarrow & (q+1), \\
		(q-1 \,: \, 1) & \leftrightarrow & (q), \\
		(q \,: \, q-1) & \leftrightarrow & (2q-1), \\
		(1 \,: \, 2q-1) & \leftrightarrow & (2q), \\
		(q-1 \,: \, 2) & \leftrightarrow & (q+1), \\
		(2 \,: \, q-1) & \leftrightarrow & (q+1), \\
		(q-3 \,: \, 2) & \leftrightarrow & (q-1), \\
		(2 \,: \, q-3) & \leftrightarrow & (q-1), \\
		\ldots && \ldots ,\\
		(2 \,: \, 2) & \leftrightarrow & (4) .
	\end{array} $$ 
	For the basis elements of $H_1 (\operatorname{M}(2,q,2kq))$ this translates to
	$$ \begin{array}{lcl}
		a_{v-1}^{-1} & \simeq & \langle v, v-2, v-1, v \rangle \\
		& \simeq & \langle v, 0, v-2, v-1, v \rangle \\
		& \simeq & \langle 0, v-2, v-1, v, 0 \rangle
	\end{array} $$ 
	and
	$$ \begin{array}{lcl}
		a_{v} & \simeq & \langle v+1, v, v-1, v+1 \rangle \\
		& \simeq & \langle v+1, v+q+1, v, v-1, v+1 \rangle \\
		& \simeq & \langle v+1, v+q+1, v+2, v, v-1, v+1 \rangle \\
		& \simeq & \langle v+1, v+q+1, v+2, 0, v, v-1, v+1 \rangle \\
		& \simeq & \langle 0, v, v-1, v+1, v+q+1, v+2, 0 \rangle
	\end{array} $$ 
	for $v \in \{ 2,4, \ldots , q-1 \}$, and thus
	$$ \begin{array}{lcl}
		a_{v-1}^{-1} a_{v} & \simeq & \langle 0, v-2, v-1, v, 0 \rangle \langle 0, v, v-1, v+1, v+q+1, v+2, 0 \rangle \\
		& \simeq & \langle 0, v-2, v-1, v+1, v+q+1, v+2, 0 \rangle \\
		& \simeq & \langle 0, v-2, v-1, v+1, v, v+q+1, v+2, 0 \rangle \\
		& \simeq & \langle 0, v-2, v-1, v+1, v, v+2, 0 \rangle \\
		& \simeq & \langle 0, v-2, v-1, v+1, v, 0 \rangle .
	\end{array} $$ 

	\medskip
	Putting these pieces together this results in
	$$ \begin{array}{lcl}
		(a_{1}^{-1} a_{2}) \ldots (a_{q-2}^{-1} a_{q-1}) & \simeq & \langle 0, 0, 1,3,2 ,0 \rangle \langle 0, 2, 3,5,4 ,0 \rangle \ldots \langle 0, v-2, v-1, v+1, v, 0 \rangle \\
		& \simeq & \langle 0, 1,3,5, \ldots , q-2,q,q-1,0 \rangle \\
		& \simeq & \langle 0, 1,q,q-1,0 \rangle \\
		& \simeq & \langle 0, 1,q+1,q,q-1,0 \rangle \\
		& \simeq & \langle 0, q+1,q,q-1,q+1,0 \rangle \\
		& \simeq & \langle q+1,q,q-1,q+1 \rangle \\
		& \simeq & g \cdot a_{q-1} . \hfill \qedhere
	\end{array} $$ 
\end{proof}

\section{The topological types of $\operatorname{M}(p,q,r)$}
\label{sec:topTypes}

In this section we prove Theorem \ref{thm:brieskornSpheres}. 
That is, we show that $\operatorname{M}(p,q,r)$ is homeomorphic
to the Seifert fibred spaces of type 
	$$ \operatorname{SFS} [ (\mathbb{T}^2 )^{\# (a-1)(b-1)/2} : (-p/a,b_1)^b 
  (q/b,b_2)^a (-r/(ab) ,b_3) ] $$
with $a = \gcd(p,r)$ and $b = \gcd(q,r)$, $r > 0$.

\medskip
In particular, we show that $\operatorname{M}(p,q,r)$ is homeomorphic to 
the Brieskorn homology sphere $\Sigma (p,q,r)$ whenever $p$, $q$ and $r$ are 
co-prime, $\operatorname{M}(2,q,2)$ is homeomorphic to the lens space 
$\operatorname{L}(q,1)$, and that, in the limit case $r=0$, we have
$\operatorname{M}(p,q,0) \cong (\mathbf{S}^2 \times \mathbf{S}^1)^{\# (p-1)(q-1)}$.

The proof is given as a corollary of the following five observations.

\begin{enumerate}
	\item The Seifert fibrations given in Theorem \ref{thm:brieskornSpheres} are 
    well-defined, i.e., all invariants of Seifert fibred spaces satisfying 
    the conditions of the theorem for the same triple $(p,q,r)$, $r>0$,
		are isomorphic (cf. Lemma \ref{lem:isomorphic}).
	\item $\operatorname{M}(p,q,r)$ is a combinatorial manifold for all 
     $p,q \in \mathbb{N}$, $p$ and $q$ co-prime, and for all non-negative
    integers $r\geq 0$ (cf. Lemma \ref{prop:Mpqr}).
	\item For $a = \operatorname{gcd}(p,r)$ and $b = \operatorname{gcd}(q,r)$,
		\begin{itemize}
			\item $\operatorname{F}_1(p,q,r)$ is a triangulation of $b$ disjoint 
        copies of a solid torus, 
			\item $\operatorname{F}_2(p,q,r)$ is a triangulation of $a$ disjoint 
        copies of a solid torus,
			\item for $r>0$, $\operatorname{F}_3(p,q,r)$ is a triangulation of a 
        single solid torus, and,
      \item for $r = 0$, a collection of $pq$ 
        tetrahedra glued together along edges forming a solid torus pinched
        along edges.
		\end{itemize}
		Furthermore, the boundary of the meridian disc of each torus can be 
    explicitly described (cf. Lemma \ref{lem:fibresMerDiscs}).
	\item For $r>0$, $\operatorname{B}(p,q,r)$ united with a small neighbourhood 
      of the 
      boundaries of $\operatorname{F}_i(p,q,r)$, $1 \leq i \leq 3$, is 
      homeomorphic to the Cartesian product of a circle with the orientable 
      surface of genus $\frac12 (a-1)(b-1)$ with $b+a+1$ discs removed, where 
      each of the $b+a+1$ boundary components corresponds to one
  		boundary torus of $\operatorname{F}_i(p,q,r)$, $1 \leq i \leq 3$ 
      (cf. Lemma \ref{lem:Bpqr}). In addition, the boundary curves of the 
      $b+a+1$ meridian discs of $\operatorname{F}_i(p,q,r)$, $1 \leq i \leq 3$,
  		in $\operatorname{B}(p,q,r)$ can be determined to be of the desired type 
      (cf. Lemma \ref{lem:bdryCurves}).
  \item For $r=0$, $\operatorname{B}(p,q,0) \cup \operatorname{F}_3(p,q,0)$
      united with a small neighbourhood of the 
      boundaries of $\operatorname{F}_i(p,q,0)$, $1 \leq i \leq 2$, minus a 
      small neighbourhood of $\operatorname{F}_3(p,q,0)$, is 
      homeomorphic to the Cartesian product of a circle with the orientable 
      surface of genus $\frac12 (p-1)(q-1)$ with $p+q+1$ discs removed.
      The boundary curves of the 
      $p+q$ meridian discs of $\operatorname{F}_i(p,q,0)$, $1 \leq i \leq 2$,
  		and the meridian disc of a thickened version of 
      $\operatorname{F}_3(p,q,0)$ in $\operatorname{B}(p,q,0)$ can be 
      determined to be of the desired type.
\end{enumerate}

We first give detailed proofs of these five observations before we summarise 
them in order to prove Theorem~\ref{thm:brieskornSpheres}.

\begin{lemma}
	\label{lem:isomorphic}
	Given positive integers $p,q,r \in \mathbb{N}$, $2 \leq p < q$ co-prime, 
  $r > 0$, $a = \operatorname{gcd}(p,r)$ and $b = \operatorname{gcd}(q,r)$, 
  then all Seifert fibrations
	$$ \operatorname{SFS} [ (\mathbb{T}^2 )^{\# (a-1)(b-1)/2} : (-p/a,b_1)^b 
    (q/b,b_2)^a (-r/ab ,b_3) ] $$
	satisfying
	$$ (\frac{b_1}{p} - \frac{b_2}{q} + \frac{b_3}{r} ) \frac{pqr}{ab} = 1$$
	are isomorphic. In particular, their underlying manifolds are homeomorphic.
\end{lemma}

\begin{proof}
	The isomorphism type of a Seifert fibred space with exceptional fibres
	$(a_i,b_i)$, $1\leq i \leq r$, does not change by simultaneously adding 
	$a_j$ to $b_j$ and subtracting $a_k$ from $b_k$ for any pair of indices 
	$1 \leq j,k \leq r$, or by changing the sign of all the $a_i$, $q \leq i \leq r$ 
	(cf. \cite{Orlik72SeifertMflds}).

	\medskip
	Now let $p$, $q$, $r$ be fixed and $b_i$, $b'_i$, $1\leq i \leq 3$, such that
	$$(\frac{b_1}{p} - \frac{b_2}{q} + \frac{b_3}{r} ) \frac{pqr}{ab} = 1 = 
    \frac{pqr}{ab} (\frac{b'_1}{p} - \frac{b'_2}{q} + \frac{b'_3}{r} ).$$
	In particular, this means that 
	\begin{equation}
		\label{eq:sfs}
		\frac{qr}{ab} (b_1-b'_1) - \frac{pr}{ab} (b_2-b'_2) + \frac{pq}{ab} 
    (b_3-b'_3)  = 0
	\end{equation}
	and thus
	$$\frac{r}{ab} (q (b_1-b'_1) - p (b_2-b'_2)) = \frac{-pq}{ab} (b_3-b'_3).$$
	Now note that $\operatorname{gcd} (\frac{r}{ab},\frac{-pq}{ab}) = 1$ by 
  construction and hence there exist an $\alpha \in \mathbb{Z}$ such that
	$$  (q (b_1-b'_1) - p (b_2-b'_2)) = \alpha \frac{pq}{ab}  \qquad 
    \textrm{ and } \qquad (b_3-b'_3) = \alpha \frac{r}{ab}.$$
	In particular $  q (b_1-b'_1) - p (b_2-b'_2) \equiv 0 \mod \frac{pq}{ab}$ 
  holds, and since furthermore 
  $\operatorname{gcd} (\frac{q}{b},\frac{p}{a}) = 1$, we have both 
  $q (b_1-b'_1) - p (b_2-b'_2) \equiv 0 \mod \frac{p}{a}$ and 
	$q (b_1-b'_1) - p (b_2-b'_2) \equiv 0 \mod \frac{q}{b}$ by the Chinese 
  remainder theorem.

	It follows that $(b_1-b'_1)$ is a multiple of $\frac{p}{a}$, $(b_2-b'_2)$ 
  is a multiple of $\frac{q}{b}$, $(b_3-b'_3)$ 
	is a multiple of $\frac{r}{ab}$, by Equation~(\ref{eq:sfs}) additions and 
  subtractions sum up to zero, and thus
	the Seifert fibred spaces corresponding to $(b_1,b_2,b_3)$ and 
  $(b'_1,b'_2,b'_3)$ are isomorphic.
\end{proof}

\begin{lemma}
	\label{lem:fibresMerDiscs}
	Given positive integers $p,q,r \in \mathbb{N}$, $2 \leq p < q$ co-prime, 
  $r \geq 0 $, 
	$a = \operatorname{gcd}(p,r)$ and $b = \operatorname{gcd}(q,r)$, we have:

  \begin{itemize}
    \item $ \operatorname{F}_1(p,q,r) \cong \{ 1, 2, \ldots , b \} \times 
      (\mathbf{B}^2 \times \mathbf{S}^1)$
      where the boundaries of the meridian discs $m_1^{(i)}$, $0 \leq i \leq b-1$, 
      are given by the paths
	$$ \partial m_1^{(i)} = \langle i, kq+i, 2kq+i, \ldots , (p-1)kq+i, p kq + i ,
  (k-1)pq + i, (k-2)pq +i , \ldots , i \rangle; $$
    \item $ \operatorname{F}_2(p,q,r) \cong \{ 1, 2, \ldots , a \} 
    \times (\mathbf{B}^2 \times \mathbf{S}^1),$
		where the boundaries of the meridian discs $m_2^{(j)}$, $0 \leq j \leq a-1$, 
    are given by the paths
	$$ \partial m_2^{(j)} = \langle j, mp+j, 2mp+j, \ldots , (q-1)mp+j, q mp + j , 
    (m-1)pq + j, (m-2)pq +j , \ldots , j \rangle; $$
    \item for $r > 0$, 
      $ \operatorname{F}_3(p,q,r) \cong \mathbf{B}^2 \times \mathbf{S}^1$
	  	where the boundary of the meridian disc $m_3$ is given by the path
	$$ \partial m_3 = \langle 0, pq, pq+1, pq+2, \ldots , -pq-1, -pq, 0 \rangle ;$$
    \item for the limit case $r=0$, $\operatorname{F}_3(p,q,0)$
      is a collection of $pq$ tetrahedra glued together along common edges,
      forming a solid torus pinched along $pq$ edges.
  \end{itemize}
\end{lemma}

\begin{proof}
	First let us assume that $(p-k)q \geq kq$. By definition we have
	$$ \begin{array}{rcl}
		\operatorname{F}_1 (p,q,r) &=& \{ (b_i \,:\, a_i \,:\, b_i \,:\, 
    2 p q - 2 b_i - a_i + r) \, |\,  1 \leq i \leq N((p - k)q,kq) \} \\
		&=& \{ d_i \,|\, 1 \leq i \leq N((p - k)q,kq) \}
	\end{array}$$
	where $N((p - k)q,kq)$ denotes the number of steps to compute 
  $\operatorname{gcd} ((p - k)q,kq) = q$ using the Euclidean algorithm 
	given by Equation~(\ref{eq:EuclAlgo}), and $(a_i,b_i)$ denotes the arguments 
  of the Euclidean algorithm {\em after} the $i$-th step
	(see Section \ref{sec:3paramFam} for details).

	$\partial \operatorname{F}_1 (p,q,r)$ is contained in $d_1$, and by construction $d_i$ can be 
	collapsed onto $d_{i+1}$ whenever each tetrahedron of $d_i$ 
	contains a boundary face of the complex. Hence $\operatorname{F}_1 (p,q,r)$ can be collapsed 
	onto $d_{N((p - k)q,kq)} = \{ (q : q : q : 2pq + r -3q) \}$.
	By definition, we have $\operatorname{gcd} (q,r) = b$ and hence 
	$$\begin{array}{rcl}
		\operatorname{gcd} (q,2pq + r - 3q) &=& \operatorname{gcd} (q,2pq + r - 3q) \\
		&=& \operatorname{gcd} (q,(2p-3)q + r) \\
		&=& \operatorname{gcd} (q,r) \\
		&=& b.
	\end{array} $$
	It follows that $\operatorname{F}_1 (p,q,r)$ collapses to $b$ connected components each with 
	$(2pq+r)/b$ vertices and all isomorphic to 
	$$\{ (q/b:q/b:q/b:(2pq + r -3q)/b) \} \cong \{ (1:1:1:(2pq + r)/b - 3) \}$$
	and thus 
	$$ \operatorname{F}_1 (p,q,r) \cong \{ 1, 2, \ldots , b \} \times (\mathbf{B}^2 \times \mathbf{S}^1)$$
	(see Figure \ref{fig:collScheme}).

  \medskip
  The proof for the case $(p-k)q < kq$ is completely analogous
	as is the proof that
	$$ \operatorname{F}_2 (p,q,r) \cong \{ 1, 2, \ldots , a \} \times (\mathbf{B}^2 \times \mathbf{S}^1)$$
	(see Figure \ref{fig:collScheme} again). To see that 
  $\operatorname{F}_3 (p,q,r)$ is a solid torus for $r>0$ and a collection
  of tetrahedra glued together along common edges for $r=0$, just note that it 
  coincides with the last $\lfloor \frac{r}{2} \rfloor + 1$ difference cycles 
  of the boundary complex of the cyclic polytope $\partial C_4 (2pq+r)$. 
  For more about how the boundary complex of the cyclic $4$-polytope
	can be decomposed into difference cycles, see \cite{Spreer14CyclicCombMflds}.
	\begin{figure}[h!]
		\begin{center}
			\includegraphics[width=0.9\textwidth]{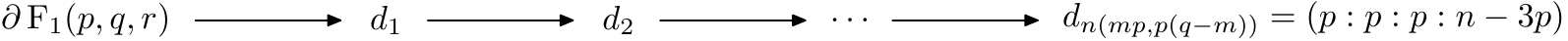}

			\vspace{.5cm}
			\includegraphics[width=0.9\textwidth]{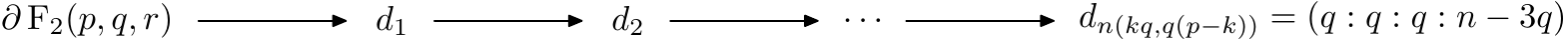} 
			\caption{$\operatorname{F}_1 (p,q,r)$ and $\operatorname{F}_2 (p,q,r)$ 
      collapsing onto multiple solid tori. \label{fig:collScheme}}
		\end{center}
	\end{figure}

	In order to prove that $\partial m_1^{(i)}$, $0 \leq i \leq b-1$, is the boundary of a meridian disc 
	of $\operatorname{F}_1 (p,q,r)$ we have to show that
	$\partial m_1^{(i)} \subset \partial \operatorname{F}_1 (p,q,r)$ is {\it i)} closed, {\it ii)} simple, 
	{\it iii)} homologous to zero inside $\operatorname{F}_1 (p,q,r)$,
	and {\it iv)} homologically non-trivial in $\partial \operatorname{F}_1 (p,q,r)$.

  \medskip
	Again, let $(p-k)q \geq kq$. The fact that {\it i)} holds follows immediately 
	from the definition. To see that {\it ii)} is true assume there is a point
  of self-intersection, that is, $x \cdot (kq) = y \cdot (pq)$ for integers 
  $0 < x \leq p$ and $0 < y \leq k$. Then
	$$ \begin{array}{rcl}
		x \cdot (kq) = y \cdot (pq) &\Leftrightarrow& x \cdot (mp-1) = y \cdot (pq)\\
		&\Leftrightarrow& x \cdot (mp) - x = y \cdot (pq)\\
		&\Leftrightarrow& x = p \cdot (x m - y q)\\
		&\Leftrightarrow& p \,\mid\, x\\
	\end{array} $$
	and since $0 < x \leq p$, the only solution is $x=p$, $y=k$, and therefore
  $\partial m_1^{(i)}$ is simple. To prove {\it iii)} 
	note that by construction we can homotopically deform $\partial m_1^{(i)}$ over triangles (that is, replace
	$\langle \ldots , v, w, \ldots \rangle$ by $\langle \ldots , v, u, w, \ldots \rangle$ if $\langle u,v,w \rangle$
	is a triangle) such that
	\small
	$$  \begin{array}{rcl}
		\langle i, kq+i, \ldots , (p-1)kq+i, k (pq) + i , (k-1)pq + i, \ldots , i \rangle &\cong& 
		\langle i, p+i, \ldots , (kq-1) \cdot p, kq \cdot p , (kq-1) \cdot p , \ldots , i \rangle \\ 
		&\cong& 0. \\
	\end{array} $$
	\normalsize
	Finally, to prove {\it iv)} we observe that $\partial m_1^{(i)}$ wraps $q/b$ times around the 
	fundamental domain of the $i$-th boundary 
	component of $\operatorname{F}_1 (p,q,r)$ 
	(cf. Figure \ref{fig:Bpqr_main}) in the horizontal direction 
	and hence cannot be homologous to zero in $\partial \operatorname{F}_1 (p,q,r)$.
	All together it follows that $m_1^{(i)}$ is a meridian disc of the $i$-th connected component of $\operatorname{F}_1 (p,q,r)$.
	Again, the proof in the case $(p-k)q < kq$ and the proof for $m_2^{(j)}$, $1 \leq j \leq a-1$, are completely analogous. 

	\medskip
	To see that $m_3$ is a meridian disc of $\operatorname{F}_3 (p,q,r)$, $r>0$,
  see Figure \ref{fig:m3} where $m_3 \subset \operatorname{F}_3 (p,q,r)$
	is given explicitly.

	\begin{figure}[h!]
		\begin{center}
			\includegraphics[width=0.9\textwidth]{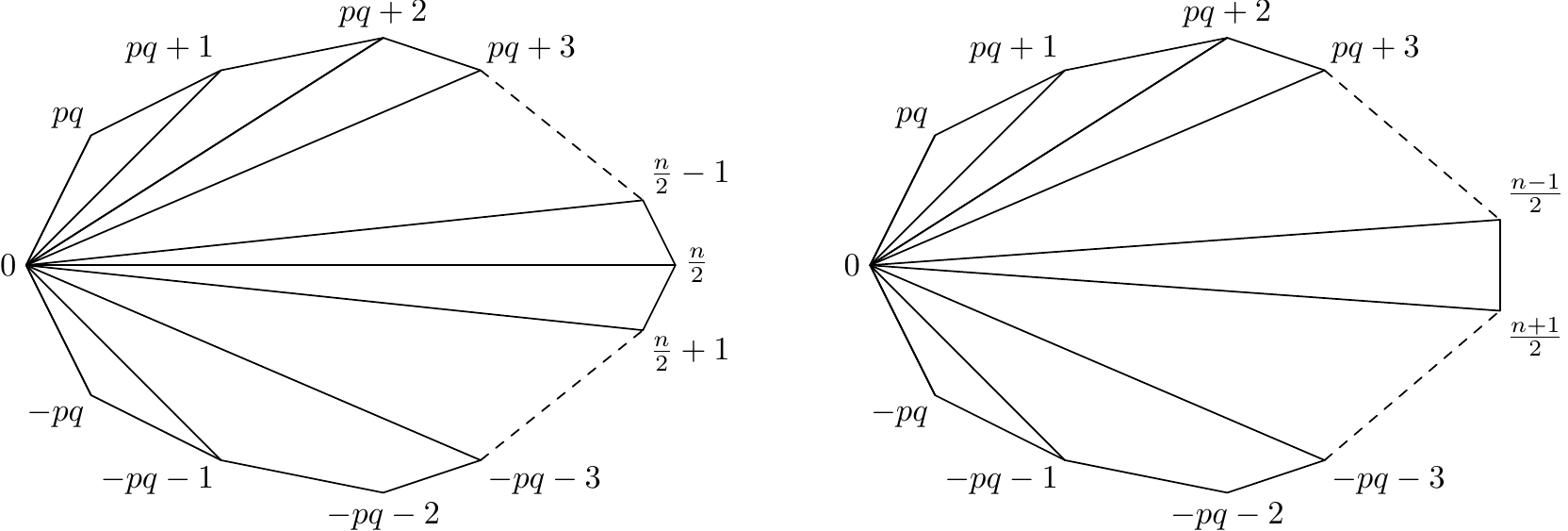} 
			\caption{The meridian disc of $\operatorname{F}_3 (p,q,r)$ for 
      $n = 2pq +r$ even (left) and odd (right). \label{fig:m3}}
		\end{center}
	\end{figure}
\end{proof}

\begin{lemma}
	\label{lem:Bpqr}
	Let $\mathscr{B}$ be a thickened version of the complex 
  $\operatorname{B}(p,q,r)$, $r>0$,
	such that $\mathscr{B}$ is orientable and all boundary components of
  $\operatorname{B}(p,q,r)$ are disjoint. Then
	$$\mathscr{B} \cong \mathbf{S}^1 \times \mathbb{S}_{\frac12 (a-1)(b-1)}^{b+a+1},$$
	where $\mathbb{S}_g^m$ is the $m$-punctured orientable surface of genus $g$.
\end{lemma}

\begin{proof}
	In essence, we read off the diagrams given in Figures~\ref{fig:Bpqr_rep} 
  and \ref{fig:Bpqr}.
	The rest of the proof consists of details and bookkeeping.

  \medskip
  $\operatorname{B}(p,q,r)$ consists of three difference 
  cycles of full length and	hence contains $3n = 6pq + 3r$ tetrahedra. These 
  split into $n$ disjoint systems	of representatives for the difference 
  cycles of $3$ tetrahedra each.  One of these systems of representatives is 
  given by
	$$ \langle 0,1,mp,pq \rangle, \langle 1,mp,pq,pq+1 \rangle \textrm{ and } 
    \langle mp,pq,pq+1,p(q+m) \rangle .$$
  Figure \ref{fig:Bpqr_rep} illustrates how these $n$ groups 
  of $3$ tetrahedra can be	stacked onto the fundamental domain of the boundary
  torus
	$$ \begin{array}{rcl}
		\partial \operatorname{F}_3 (p,q,r) &=& \{ (1 : pq-1 : r + pq ), 
      (1 : r+pq : pq-1) \} \\
		&=& \{ \langle 0,1,pq \rangle, \langle 1,pq,pq+1 \rangle, \ldots \} .
	\end{array}$$
	Here two vertically neighbouring groups are glued together along the triangles
	$\langle pq, pq+1, p(q+m) \rangle$ and their translates, and 
	the complex $\operatorname{B}(p,q,r)$ is obtained by identifying pairs of 
  vertical edges of the resulting complex 
	given in Figure \ref{fig:Bpqr} where each vertical edge of 
  $\partial \operatorname{F}_3 (p,q,r)$ (for example $\langle 0, pq \rangle$ 
	and $\langle 1, pq+1 \rangle$ in the lower left corner) is glued to the 
  unique vertical edge with the corresponding vertex labels
	not touching the fundamental domain (for example $\langle mp, p(q+m) \rangle$). 
  Note that it already follows from the 
	cyclic symmetry that exactly $n$ of these pairs of vertical edges exist.

	\begin{figure}[h!]
		\begin{center}
			\includegraphics[width=0.7\textwidth]{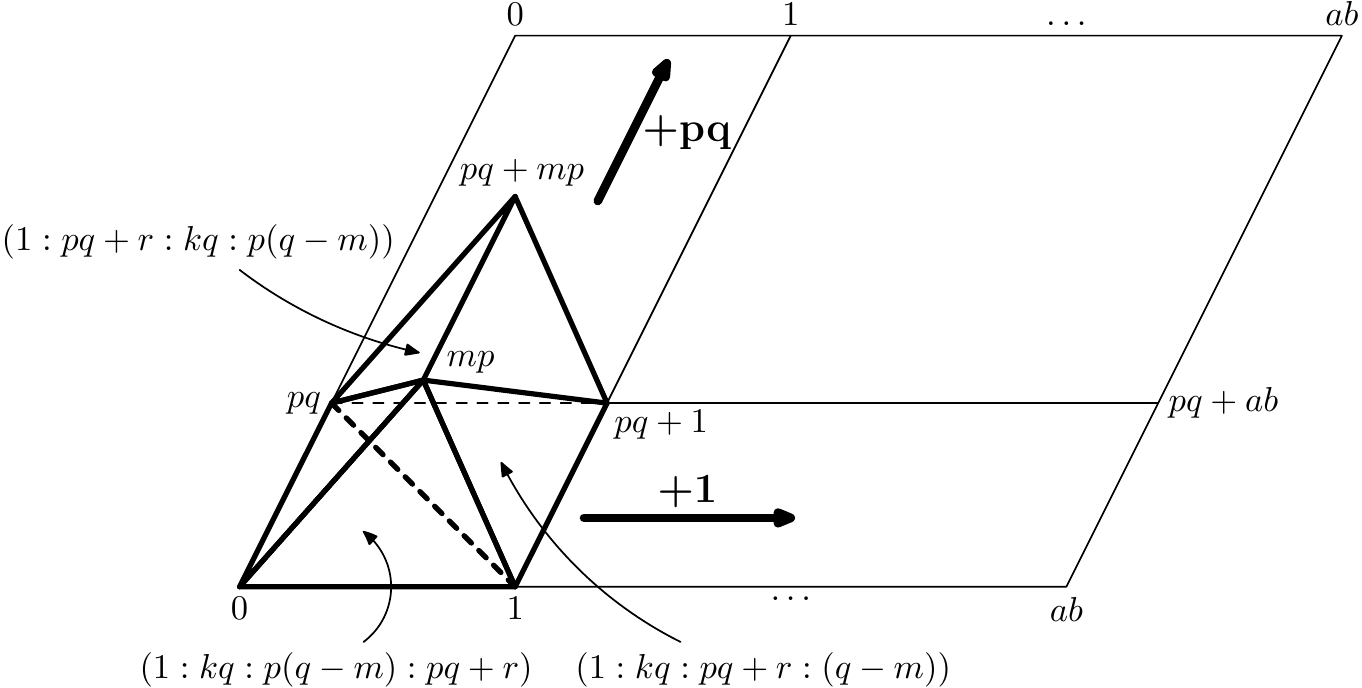} 
			\caption{System of representatives of the three difference cycles of 
        $\operatorname{B}(p,q,r)$. \label{fig:Bpqr_rep}}
		\end{center}
	\end{figure}
	This construction together with Lemma \ref{lem:fibresMerDiscs} gives rise 
  to a complex with $b + a + 1$ boundary tori where all
	the boundary tori run vertically relative to the fundamental domain given in 
  Figure \ref{fig:Bpqr}. A more schematic drawing
	of $\operatorname{B}(p,q,r)$ together with its boundary tori is given in 
  Figure \ref{fig:Bpqr_main}.

	\begin{figure}[ht]
		\begin{center}
			\includegraphics[width=0.9\textwidth]{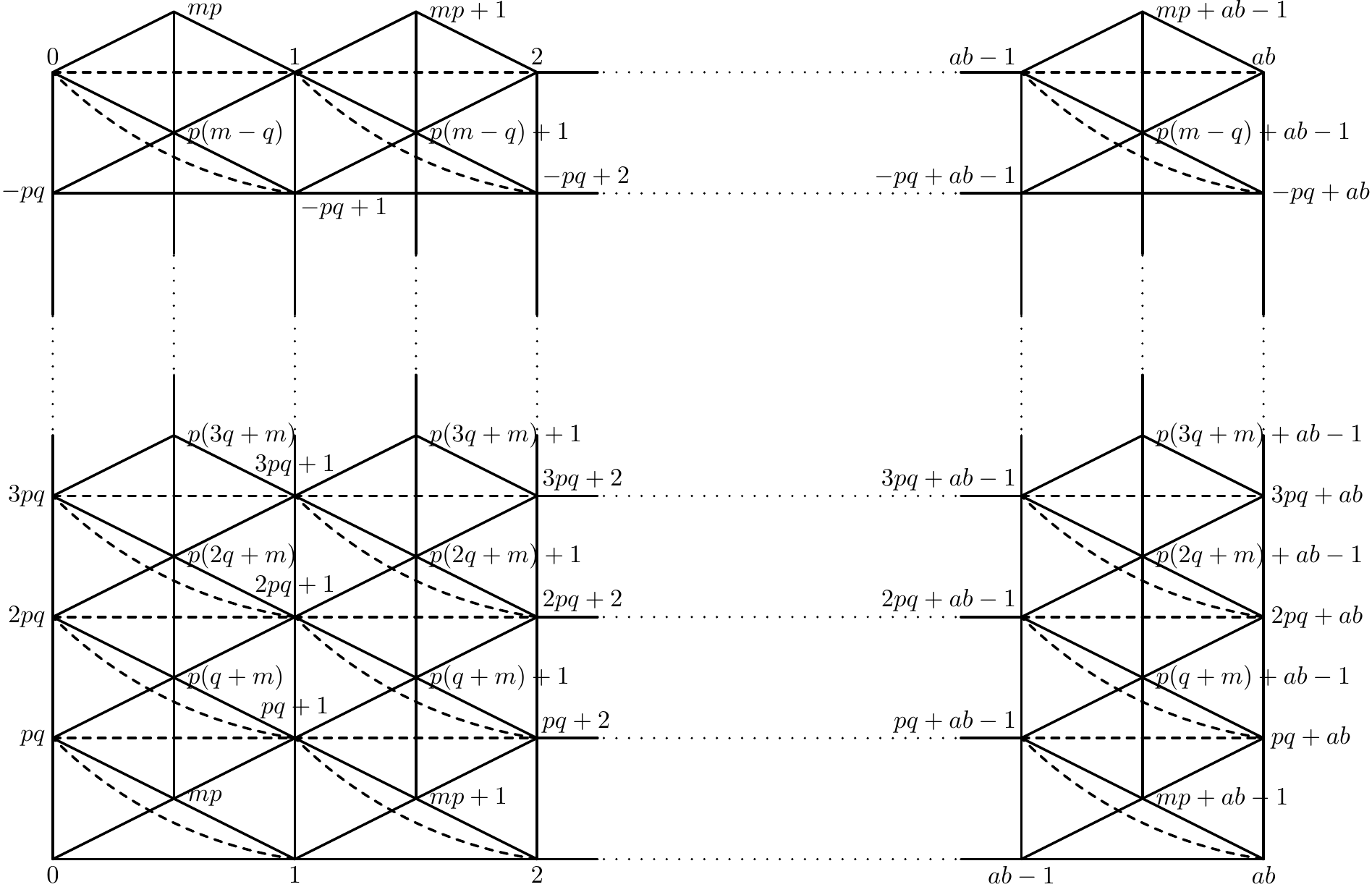} 
			\caption{$\operatorname{B}(p,q,r)$ after gluing  
      tetrahedra along common boundary faces. 
			Triangles at the top (e.g., 
      $\langle 0,1,mp \rangle$) are glued to the bottom ones. 
			Edges with equally labelled endpoints are identified. \label{fig:Bpqr}}
		\end{center}
    \vspace{-.05cm}
	\end{figure}

	This already tells us that $\operatorname{B}(p,q,r)$ is the Cartesian 
  product of a circle with $S$, where $S$ is a closed surface minus $b+a+1$ 
  discs. $S$ runs horizontally relative to the fundamental domain 
  (i.e., $S$ meets $\partial \operatorname{F}_3(p,q,r)$ 
	in a curve of the same homotopy class as the horizontal line in the 
  fundamental domain, plus a necessary vertical shift 
	at the right hand side to close it); and the circle runs vertically. 
	In order to see what $S$ looks like, we must pay attention to how exactly 
  the vertical edges of the complex
	are glued together. A basic observation exploiting the cyclic symmetry of the
	complex tells us that every vertical line in Figure \ref{fig:Bpqr} of $\operatorname{B}(p,q,r)$ contains a single
	vertex $\beta$ such that $0 \leq \beta < ab$. For the vertical lines touching the fundamental domain of $\partial \operatorname{F}_3(p,q,r)$
	these are at the very bottom except for at the rightmost line where vertex $0$ is shifted by $\alpha$, where $\alpha$ describes how the
	vertical boundary parts of the fundamental domain of $\operatorname{F}_3(p,q,r)$ are shifted in order to be glued together
	to build a torus. For the other vertical lines the unique vertex label $\beta$, $0 \leq \beta < ab$, is shifted by $\frac{n}{ab} - \gamma$ or $\frac{n}{ab} - \gamma + \alpha$,
	where $\gamma$ describes the vertical distance (modulo $\frac{n}{ab}$) of $\beta$ and vertex $mp$. Figure \ref{fig:Bpqr_section} shows 
	the cut through $\operatorname{B}(p,q,r)$ containing all of these vertices and thus resulting in a simple representation of the base surface $S$.

	Note that $S$ after identifying vertices with equal labels contains exactly 
  $b+a+1$ edge disjoint boundary circles such that each belongs to a 
	unique connected component of $\operatorname{F}_i(p,q,r)$, $1 \leq i \leq 3$. 
  Each of these boundary circles can be given an orientation such
	that each of their edges is oriented clockwise in the drawing of $S$ given 
  in Figure \ref{fig:Bpqr_section}. It follows that $S$, and 
	hence $\operatorname{B}(p,q,r)$, can be thickened to give a bounded 
  $3$-manifold $\mathscr{B}$ homeomorphic to the 
	Cartesian product of the circle with an oriented surface with $b+a+1$ 
  punctures. Note that $S$ has $ab$ vertices, $3ab$ edges and $ab$ triangles, 
  hence Euler characteristic $\chi (S) = -ab$, and we have
  $S \cong \mathbb{S}_{\frac12 (a-1)(b-1)}^{b+a+1}$.
\end{proof}

\begin{lemma}
	\label{lem:Bpq0}
	Let $\mathscr{B}$ be a thickened version of 
  $\operatorname{B}(p,q,0) \cup \operatorname{F}_3(p,q,0)$, with
  a slightly thickened version of $\operatorname{F}_3(p,q,0)$ drilled out,
	such that $\mathscr{B}$ is orientable and all boundary components of
  $\operatorname{B}(p,q,0)$ are disjoint. Then
	$$\mathscr{B} \cong \mathbf{S}^1 \times \mathbb{S}_{\frac12 (p-1)(q-1)}^{p+q+1},$$
	where $\mathbb{S}_g^m$ is the $m$-punctured orientable surface of genus $g$.
\end{lemma}

\begin{proof}
  The proof is largely analogous to the proof of Lemma~\ref{lem:Bpqr} with
  some minor changes. 

  Since $r=0$ we have $a = \gcd (p,0) = p$ and
  $b = \gcd (q,0) = q$. Hence Figure~\ref{fig:Bpqr} has only
  two rows (note that $pq + ab = 2pq = 0$), where the top-row is identified with
  the bottom row by folding them up, leaving $pq$ tetrahedron-shaped holes
  with boundaries of type
  $$ \left \langle \langle \ell, \ell+1 , pq + \ell \rangle, 
     \langle \ell, \ell+1 , pq + 1+ \ell \rangle, 
     \langle 0, pq+\ell, pq+1+ \ell \rangle, 
     \langle 1, pq+\ell, pq+1+ \ell \rangle \right \rangle$$
  for $0 \leq \ell \leq pq-1$ which, in $\operatorname{M}(p,q,0)$, are filled
  with the $pq$ tetrahedra
  of $\operatorname{F}_3(p,q,0) = \{ (1\,:\, pq-1 \,:\, 1 \,:\, pq-1) \}$.
  Drilling out a slightly thickened version of $\operatorname{F}_3(p,q,0)$
  leaves us with a torus boundary component on the bottom of 
  Figure~\ref{fig:Bpqr} (as long as $\operatorname{B}(p,q,0)$ has 
  been sufficiently thickened before near the edges $\langle \ell, pq + \ell \rangle$, 
  $0\leq \ell \leq pq-1$).
	\begin{figure}[h!]
		\begin{center}
			\includegraphics[width=\textwidth]{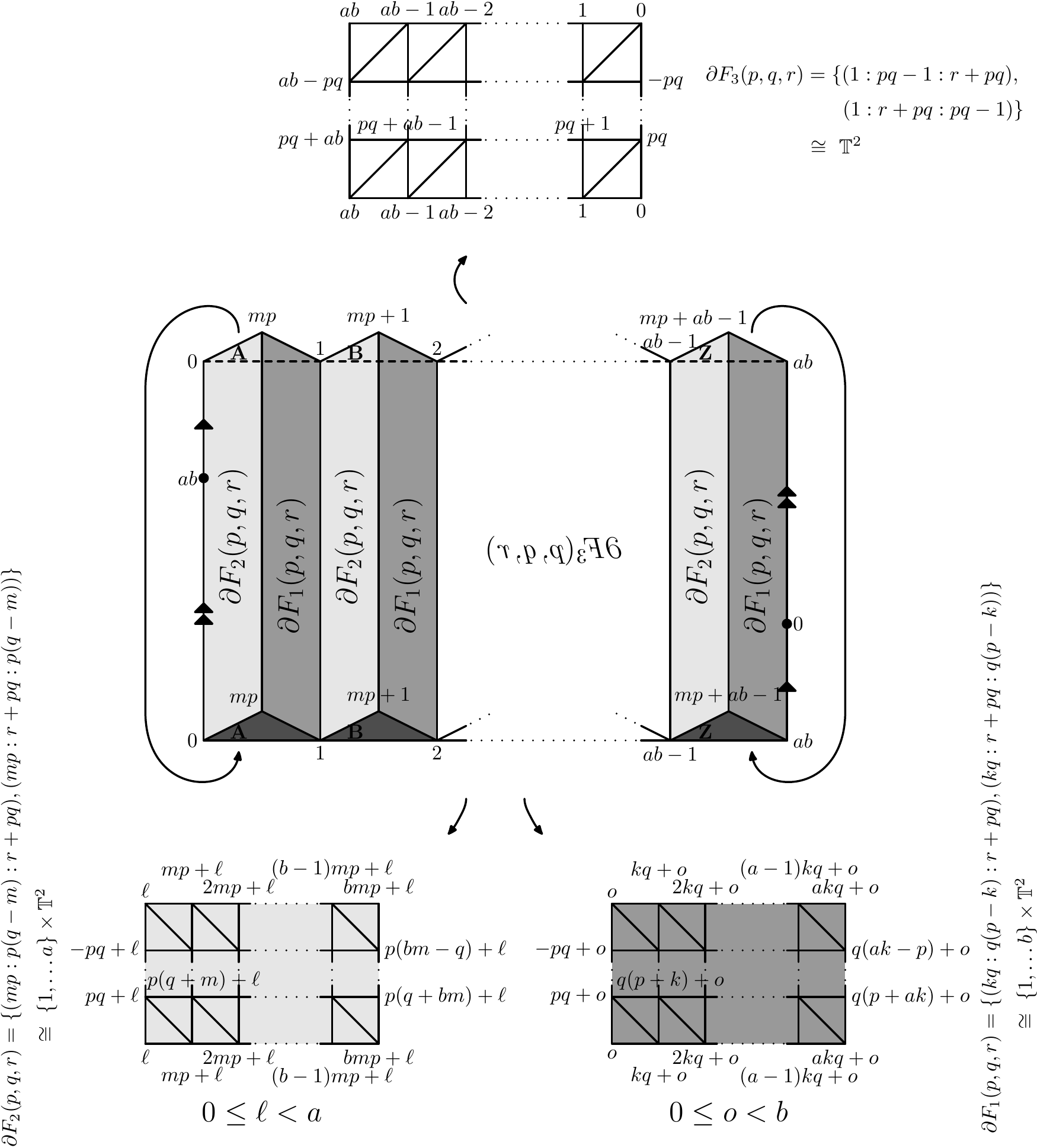} 
			\caption{The $(a+b+1)$ boundary tori of $\operatorname{B}(p,q,r)$. 
        \label{fig:Bpqr_main}}
		\end{center}
	\end{figure}

  Hence we get a space with $a+b+1 = p+q+1$ boundary tori. A two-row version of 
  Figure \ref{fig:Bpqr_main} shows the complex before thickening and drilling.
  As in the proof of Lemma~\ref{lem:Bpqr} the boundary tori run 
  vertically relative to the fundamental domain.

  This tells us that $\mathscr{B}$ is the Cartesian 
  product of a surface minus $p+q+1$ disjoint discs $S$ with a circle,
	where $S$ is running horizontally relative to the fundamental domain.
  The hypothesis now follows analogously with the shifts $\alpha = -1$,
  $\beta = mp$ and $\gamma = 0$.
\end{proof}

\begin{lemma}
	\label{lem:bdryCurves}
	Relative to the fundamental domain and base orbifold chosen in Figure 
  \ref{fig:Bpqr_section}, the types of the
	exceptional fibres for $r>0$ are 
	\begin{itemize}
		\item $(-\frac{p}{a}, \frac{ab}{n}(p \gamma  - 
      \frac{p \alpha(\beta-1)}{ab}- k))$ for the $b$ exceptional fibres of 
      $\operatorname{F}_1(p,q,r)$, 
		\item $(\frac{q}{b}, \frac{ab}{n}(q \gamma  - 
      \frac{q \alpha \beta}{ab}- m))$ for the $a$ exceptional fibres of 
      $\operatorname{F}_2(p,q,r)$, and
		\item $(-\frac{r}{ab}, \frac{ab}{n}(2-\frac{\alpha r}{ab}))$ for the 
      exceptional fibre of $\operatorname{F}_3(p,q,r)$, 
	\end{itemize}
	where 
	\small
	$$\alpha = -\left ( \left ( \frac{pq}{ab} \right )^{-1} \mod \frac{n}{ab} \right ), 
    \qquad \beta = mp \mod ab \quad \textrm{ and } \quad 
    \gamma = \left ( \frac{pq}{mp-\beta} \right )^{-1} \mod \frac{n}{ab}$$
	\normalsize
	are the shifts defining the identifications in $\operatorname{B} (p,q,r)$ as 
  shown in Figure \ref{fig:Bpqr_section}.

  \medskip
  For $r=0$ we get $q$ fibres of type $(-1,0)$, $p$ fibres of type $(1,0)$, and
  one fibre of type $(0,1)$.
\end{lemma}

\begin{proof}
	Again, we proof the statement by looking at Figures~\ref{fig:Bpqr_main} 
  and \ref{fig:Bpqr_section}.

	\medskip
	An exceptional fibre is of type $(a,b)$ if the meridian disc of its solid torus neighbourhood 
	is glued to a closed curve in the corresponding boundary torus of $\operatorname{M}(p,q,r)$ 
	which wraps $a$ times around the torus in the direction of $S$
	(this is referred to as {\em the horizontal direction}) and $b$ times 
  in the direction of the fibres ({\em the vertical direction}).

	In order to determine the exact types of exceptional fibres in $\operatorname{M} (p,q,r)$ 
	we have to specify exactly how the vertical lines in the
	fundamental domain of $\operatorname{B} (p,q,r)$ (as shown for example in Figure 
	\ref{fig:Bpqr_main}) are identified. First of all the top boundary
	$\langle 0, 1, \ldots , ab \rangle$ is identified with the bottom boundary 
	$\langle 0, 1, \ldots , ab \rangle$  without any shift and as 
	indicated by the vertex labels. The left boundary $\langle 0, pq, 2pq, \ldots , 0 \rangle$ 
	is identified with the right boundary 
	$\langle ab, pq+ab, 2pq+ab, \ldots , ab \rangle$ by shifting the right boundary {\em down} 
	by $\alpha$ rows (cf. \ref{fig:Bpqr_section}).
	Now the vertical lines of type 
	$\langle mp, pq+mp, 2pq+mp, \ldots , mp \rangle \subset \partial \operatorname{F}_i(p,q,r)$, $1 \leq i \leq 2$, 
	are glued to 
	their counterparts in $\partial \operatorname{F}_3(p,q,r)$ by shifting them $\beta$ columns 
	to the {\em right} and $\gamma$ rows {\em up}  (cf. \ref{fig:Bpqr_section}).

  Finally, we assign a positive orientation to all fibres that run from the bottom to the 
	top of the fundamental domain and to all horizontal paths which run from the left to
	the right on the front (boundary components of $\operatorname{F}_1(p,q,r)$ and $\operatorname{F}_2(p,q,r)$) 
	and hence from the right to the left in the back 
	($\partial \operatorname{F}_3(p,q,r)$) of $\operatorname{B}(p,q,r)$.

  \medskip
	Following this framework note that the boundary curves of the meridian discs $\partial m_1^{(i)}$, $0 \leq i \leq b-1$, 
	have exactly length $p$ in the horizontal direction, and that the fundamental domains of the corresponding boundary tori 
	(cf. Figure \ref{fig:Bpqr_main}) have exactly $a$ columns. Furthermore, $\partial m_1^{(i)}$ runs from the right 
	to the left and hence $\partial m_1^{(i)}$ wraps around the fundamental domain of $\operatorname{B}(p,q,r)$ exactly $-p/a$ times
	in the horizontal direction. Using the same reasoning we can see that $\partial m_2^{(j)}$, $0 \leq j \leq a-1$, wraps around the 
	fundamental domain of $\operatorname{B}(p,q,r)$ exactly $q/b$ times in the horizontal direction and $\partial m_3$ exactly $-r/(ab)$ times.

	To determine how often the boundary curves of the meridian discs wrap around the fundamental domain in the vertical direction,
	we have to carefully take into account the shifts $\alpha$, $\beta$ and $\gamma$ of the identifications of the vertical lines
	in $\operatorname{B}(p,q,r)$ (see Figure \ref{fig:Bpqr_section} for details).

	\begin{landscape}
	\begin{figure}[p]
		\begin{center}
			\includegraphics[width=1.2\textwidth]{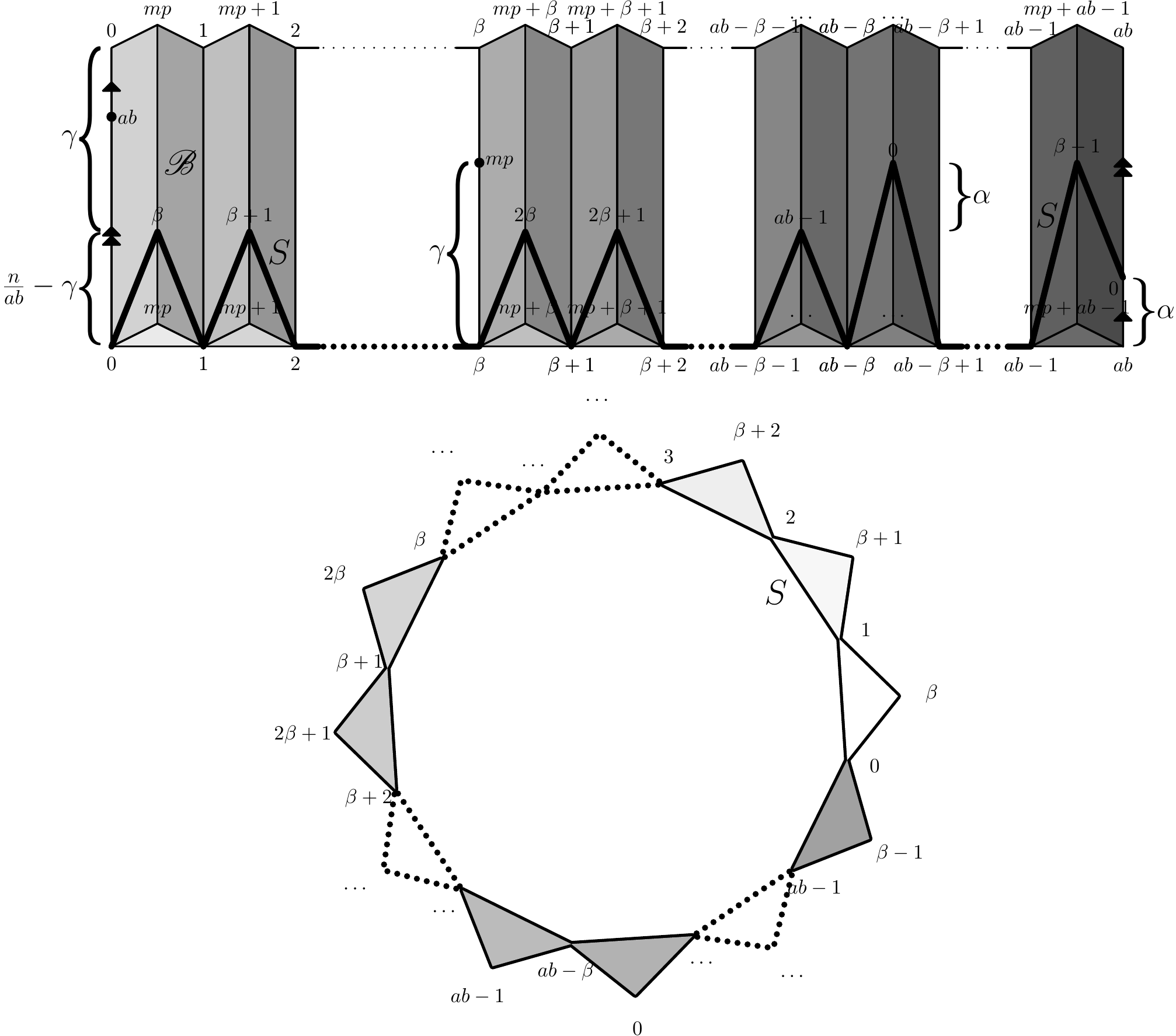} 
			\caption{Position of the base surface $S$ in the trivial $\mathbf{S}^1$-bundle $\mathscr{B}$. \label{fig:Bpqr_section}}
		\end{center}
	\end{figure}
	\end{landscape}

	The boundary curves of the meridian discs $\partial m_1^{(i)}$, $0 \leq i \leq b-1$, 
	have length $-k$ in the vertical direction and are shifted $p$ times 
	in the positive vertical direction by $\gamma$ rows. In addition to this, $\partial m_1^{(i)}$ runs 
	$p$ times half-columns in the negative horizontal direction followed by a shift 
	of $\beta - 1/2$ columns in the positive horizontal direction. This results in $p$ times a horizontal shift of $\beta -1$
	columns and for each horizontal shift of $ab$ columns in positive direction we have to add another vertical
	shift of $\alpha$ rows in the negative direction. In other words, $\partial m_1^{(i)}$ is shifted in the positive 
	vertical direction by exactly
	$$ p \cdot \gamma - \frac{p \alpha (\beta-1)}{ab} - k $$
	rows. The fact that all fundamental domains consist of $\frac{n}{ab}$ rows then proves the result.

	The vertical shifts of $\partial m_2^{(j)}$, $0 \leq j \leq a-1$, and $\partial m_3$ are computed in
	an analogous fashion. All together the exceptional fibres are as stated.

  \medskip
  For the case $r=0$, note that $pq/(ab) = 1$, $n/(ab) = 2$, and $mp < pq$. 
  Hence we get $\alpha = -1$, $\beta = mp$, $\gamma = 0$, $p+q$ fibres of
  type $(\pm 1, 0)$ and one fibre of type $(0,1)$.
\end{proof}

\begin{lemma}
  \label{lem:r0}
  Let $M = \mathbb{S}_g \times \mathbf{S}^1$ be the trivial 
  $\mathbf{S}^1$-bundle over the orientable surface of genus $g$, and let
  $M'$ be obtained from $M$ by performing surgery 
  of type $(0,1)$ along the $\mathbf{S}^1$-fibre. Then 
  $M' \cong (\mathbf{S}^2 \times \mathbf{S}^1)^{\# 2g}$.
\end{lemma}

\begin{proof}
We start by representing $M$ as a product of a $4g$-gon with opposite edges
identified and a circle $\mathbf{S}^1$. Let $T_1 \subset M$ be a solid torus
$T_1 = \mathbf{D}^2 \times \mathbf{S}^1$ inside $M$ where the first factor
$\mathbf{D}^2$ is a disc inside the $4g$-gon and the second factor 
$\mathbf{S}^1$ is a copy of the fibre of $M$; see Figure~\ref{fig:4ggon}
for a picture of $M \setminus T_1$.
Now let $T_2 = \mathbf{S}^1 \times \mathbf{D}^2$ be a solid torus, and let
$M' = ( M \setminus T_1 ) \cup T_2$ such that the $\mathbf{S}^1$-factor of $T_2$ is 
glued to the boundary of the $\mathbf{D}^2$-factor of $T_1$, and the boundary of
the $\mathbf{D}^2$-factor of $T_2$ is glued to a copy of the $\mathbf{S}^1$-factor of $T_1$
on the boundary of $M \setminus T_1$. In other words, $M'$ is obtained by
performing surgery in $M$ of type $(0,1)$ along the fibre. Denote by
$i : T_2 \to M'$ the embedding of $T_2$ into $M'$ defined by this surgery. 
\begin{figure}
	\begin{center}
		\includegraphics[width=0.7\textwidth]{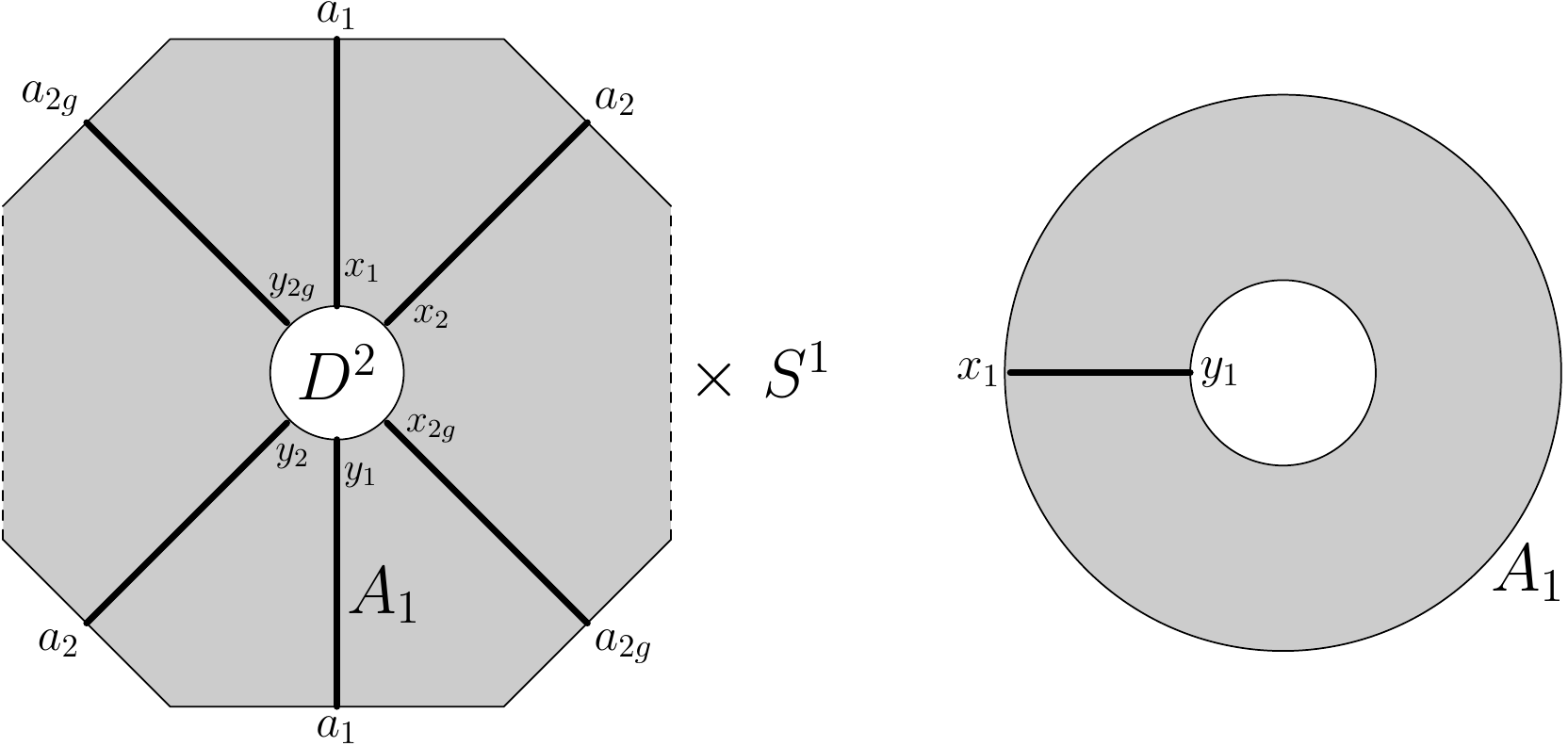} 
		\caption{On the left: the trivial circle bundle over a punctured orientable
      genus $g$ surface $M \setminus T_1$. On the right: a fibre cross interval 
      of $M\setminus T_1$. \label{fig:4ggon}}
	\end{center}
\end{figure}

It follows that there exist $4g$ disjoint disks of type 
$\{ x_i \} \times \mathbf{D}^2$ and $\{ y_i \} \times \mathbf{D}^2$,
$1 \leq i \leq 2g$, with $x_i$, $y_i$ as indicated in Figure~\ref{fig:4ggon},
which close off the $2g$ disjoint annuli $A_i = [ x_i , y_i ] \times \mathbf{S}^1$
inside $M'$, yielding $2g$ (simultaneously) non-separating disjoint 
two-spheres inside $M'$. To see why the spheres are non-separating, note that 
all corners of the $4g$-gon are identified in $M'$ and every piece of $M'$ 
after cutting out the spheres is still connected to one of the corners.

Now cutting along all $2g$ of these $2$-spheres yields $2g$ pieces of type 
$( \mathbf{S}^2 \times \mathbf{S}^1 ) \setminus \mathbf{D}^3$ and
a $3$-sphere with $2g$ punctures. Hence $M'$ is homeomorphic to a connected
sum of type $( \mathbf{S}^2 \times \mathbf{S}^1 )^{\# 2g}$.
\end{proof}

With these building blocks in mind we can now finish the proof of Theorem \ref{thm:brieskornSpheres}.

\begin{proof}[Proof of Theorem \ref{thm:brieskornSpheres}]
  Let $r>0$. First of all, by Lemma \ref{lem:isomorphic} the topological type of 
  $\operatorname{M}(p,q,r)$ as stated in
	Theorem \ref{thm:brieskornSpheres} is unique and thus well-defined.

	\medskip
	Now, by Lemma \ref{lem:Bpqr} we know that a thickened version $\mathscr{B}$ 
  of $\operatorname{B}(p,q,r)$ is homeomorphic
	to $\mathbf{S}^1 \times \mathbb{S}_{\frac12 (a-1)(b-1)}^{b+a+1}$. We construct 
  $\mathscr{B}$ by gluing a small neighbourhood
	of the boundary of each of the $b+a+1$ connected components of 
  $\operatorname{F}_i(p,q,r)$, $1 \leq i \leq 3$, to $\operatorname{B}(p,q,r)$.
	This results in the space $\mathbf{S}^1 \times \mathbb{S}_{\frac12 (a-1)(b-1)}^{b+a+1}$ 
  as, by Lemma \ref{lem:fibresMerDiscs} each of these components 
	is a solid torus.

	By Lemma \ref{lem:bdryCurves}, the boundary curves of the meridian discs 
  of these solid tori are of type 
	\begin{itemize}
		\item $(-\frac{p}{a}, b_1)$ for the $b$ exceptional fibres of 
      $\operatorname{F}_1(p,q,r)$ for $b_1 = \frac{ab}{n}(p \gamma  - 
      \frac{p \alpha(\beta-1)}{ab}- k)$,
		\item $(\frac{q}{b}, b_2)$ for the $a$ exceptional fibres of 
       $\operatorname{F}_2(p,q,r)$ for $b_2 = \frac{ab}{n}(q \gamma  - 
      \frac{q \alpha \beta}{ab}- m)$, and
		\item $(-\frac{r}{ab}, b_3)$ for the exceptional fibre of 
      $\operatorname{F}_3(p,q,r)$ for $b_3 = \frac{ab}{n}(2-
      \frac{\alpha r}{ab})$.
	\end{itemize}
	
	Note that the number of exceptional fibres is correct and 
	changing the signs of the indices in the horizontal direction of all 
	exceptional fibres simultaneously results in the desired values $p/a$, $-q/b$ and $r/(ab)$
	but only reverses the orientation of the Seifert fibration. Thus it remains to show that 
	$$\begin{array}{rcl}
	(\frac{b_1}{p} - \frac{b_2}{q} + \frac{b_3}{r}) \frac{pqr}{ab} &=& \frac{qrb_1 - prb_2 + pqb_3}{ab} \\
	&=& \frac{qr(\frac{ab}{n}(p \gamma  - \frac{p \alpha(\beta-1)}{ab}- k)) - pr(\frac{ab}{n}(q \gamma  - \frac{q \alpha \beta}{ab}- m)) + pq(\frac{ab}{n}(2-\frac{\alpha r}{ab}))}{ab} \\
	&=& \frac{qr(p \gamma  - \frac{p \alpha(\beta-1)}{ab}- k) - pr(q \gamma  - \frac{q \alpha \beta}{ab}- m) + pq(2-\frac{\alpha r}{ab})}{n} \\
	&=& \frac{pqr(\gamma  - \gamma + \frac{q \alpha \beta}{ab} - \frac{q \alpha \beta}{ab} + \frac{\alpha}{ab} - \frac{\alpha}{ab}) + rmp - rkq + 2pq)}{n} \\
	&=& \frac{r(mp - kq) + 2pq)}{n} \\
	&=& 1 , \\
	\end{array}$$
	which proves Theorem \ref{thm:brieskornSpheres} for $r>0$.

  \medskip
	Let $r=0$. By Lemma~\ref{lem:Bpq0},
  $\operatorname{M}(p,q,0)$ can be obtained from the Cartesian product 
  $\mathbf{S}^1 \times \mathbb{S}_g^0$, $g={\frac12 (p-1)(q-1)}$, by 
  performing $p+q+1$ surgeries along the $\mathbf{S}^1$ component. By 
  Lemma~\ref{lem:bdryCurves}, all but one of these surgeries are of trivial type 
  $(\pm 1, 0)$ and do not change the topology. Hence $\operatorname{M}(p,q,0)$
  is obtained from $\mathbf{S}^1 \times \mathbb{S}_g^0$ by a single surgery 
  along $\mathbf{S}^1$ of type
  $(0,1)$, i.e., by drilling out a solid
  torus along the $\mathbf{S}^1$ component and gluing it back in with 
  meridian and longitude interchanged. By Lemma~\ref{lem:r0}
  we thus have
  $$ \operatorname{M}(p,q,0) = (\mathbf{S}^2 \times \mathbf{S}^1)^{\#(p-1)(q-1)}. $$
\end{proof}

\section{Acknowledgements}

This work was supported by the Australian Research Council under the 
Discovery Projects funding scheme, project DP1094516. Furthermore, 
the authors want to thank Wolfgang K\"uhnel and the anonymous referee 
for many valuable comments.

	{\footnotesize
	 \bibliographystyle{abbrv}
	 \bibliography{/home/jonathan/bibliography/bibliography}
	}

\bigskip
\noindent
Benjamin A.~Burton \\
School of Mathematics and Physics, The University of Queensland \\
Brisbane QLD 4072, Australia \\
(bab@maths.uq.edu.au)

\bigskip
\noindent
Jonathan Spreer \\
School of Mathematics and Physics, The University of Queensland \\
Brisbane QLD 4072, Australia \\
(j.spreer@uq.edu.au)

\end{document}